\documentclass[12pt,thmsa]{article}
\usepackage{amssymb}
\usepackage{amsfonts}
\usepackage{amsmath}

\newtheorem{theorem}{Theorem}[section]

\newtheorem{proposition}[theorem]{Proposition}

\newtheorem{remark}{Remark}[section]

\newtheorem{lemma}[theorem]{Lemma}

\newenvironment{proof}[1][Proof]{\noindent\textbf{#1.} }{\ \rule{0.5em}{0.5em}}

\begin{document}

\title{The effect of nonlocal term on the superlinear
Kirchhoff type equations in $\mathbb{R}^{N}$\thanks{
J. Sun was supported by the National Natural Science Foundation of China
(Grant No. 11671236). T.F. Wu was supported in part by the Ministry of
Science and Technology and the National Center for Theoretical Sciences,
Taiwan.}}
\author{Juntao Sun$^{a,b}$\thanks{%
E-mail address: jtsun@sdut.edu.cn(J. Sun)}, Tsung-fang Wu$^{c}$\thanks{%
E-mail address: tfwu@nuk.edu.tw (T.-F. Wu)} \\
%EndAName
{\footnotesize $^a$\emph{School of Mathematics and Statistics, Shandong
University of Technology, Zibo 255049, PR China }}\\
{\footnotesize $^b$\emph{School of Mathematical Sciences, Qufu Normal
University, Shandong 273165, PR China}}\\
{\footnotesize $^c$\emph{Department of Applied Mathematics, National
University of Kaohsiung, Kaohsiung 811, Taiwan }}}
\date{}
\maketitle

\begin{abstract}
We are concerned with a class of
Kirchhoff type equations in $\mathbb{R}^{N}$ as follows:
\begin{equation*}
\left\{
\begin{array}{ll}
-M\left( \int_{\mathbb{R}^{N}}|\nabla u|^{2}dx\right) \Delta u+\lambda
V\left( x\right) u=f(x,u) & \text{in }\mathbb{R}^{N}, \\
u\in H^{1}(\mathbb{R}^{N}), &
\end{array}%
\right.
\end{equation*}%
where $N\geq 1,$ $\lambda>0$ is a parameter, $M(t)=am(t)+b$ with $a,b>0$ and
$m\in C(\mathbb{R}^{+},\mathbb{R}^{+})$, $V\in C(\mathbb{R}^{N},\mathbb{R}^{+})$
and $f\in C(\mathbb{R}^{N}\times \mathbb{R}, \mathbb{R})$ satisfying $\lim_{|u|\rightarrow \infty }f(x,u) /|u|^{k-1}=q(x)$ uniformly in $x\in \mathbb{R}^{N}$ for any $2<k<2^{\ast}$($2^{\ast}=\infty$ for $N=1,2$ and $2^{\ast}=2N/(N-2)$ for $N\geq 3$). Unlike most other
papers on this problem, we are more interested in the effects of the functions $m$ and $q$ on the number and behavior of solutions. By using minimax method as well as Caffarelli-Kohn-Nirenberg inequality, we obtain the existence and multiplicity of positive solutions for the above problem.
\end{abstract}

\textit{Key words:} Nontrivial solutions, Kirchhoff type equations,
Caffarelli-Kohn-Nirenberg inequality, Steep potential well.

\section{Introduction}

Consider the following nonlinear Kirchhoff type equations:
\begin{equation}
\left\{
\begin{array}{ll}
-M\left( \int_{\mathbb{R}^{N}}|\nabla u|^{2}dx\right) \Delta u+\lambda
V\left( x\right) u=f(x,u) & \text{in }\mathbb{R}^{N}, \\
u\in H^{1}(\mathbb{R}^{N}), &
\end{array}%
\right.  \tag*{$\left( K_{\lambda ,a}\right) $}
\end{equation}%
where $N\geq 1,$ $\lambda>0$ is a parameter, $M(t)=am(t)+b$ with $a,b>0$ and
$m$ being positive continuous function on $\mathbb{R}^{+}$, and $f$ is a
continuous function on $\mathbb{R}^{N}\times \mathbb{R}$ such that $%
f(x,s)\equiv 0$ for all $x\in \mathbb{R}^{N}$ and $s<0.$ We assume that the
potential $V(x)$ satisfies the following hypotheses:

\begin{itemize}
\item[$(V1)$] $V\in C(\mathbb{R}^{N},\mathbb{R}^{+})$ and there exists $%
c_{0}>0$ such that the set $\{V<c_{0}\}:=\{x\in \mathbb{R}^{N}\ |\
V(x)<c_{0}\}$ has finite positive Lebesgue measure, where $\left\vert \cdot
\right\vert $ is the Lebesgue measure;

\item[$(V2)$] $\Omega =intV^{-1}(0)$ is nonempty and has smooth boundary
with $\overline{\Omega }=V^{-1}(0).$
\end{itemize}

Kirchhoff type equations, of the form similar to Eq. $(K_{a,\lambda }),$ are
often referred to as being nonlocal because of the presence of the integral.
When $m(t)=t,$ Eq. $(K_{a,\lambda })$ is analogous to the stationary case of
equations that arise in the study of string or membrane vibrations, namely,
\begin{equation}
u_{tt}-\left( a\int_{\Omega }|\nabla u|^{2}dx+b\right) \Delta u=f(x,u),
\label{1}
\end{equation}%
where $\Omega $ is a bounded domain in $\mathbb{R}^{N}.$ As an extension of
the classical D'Alembert's wave equation, Eq. (\ref{1}) was first presented
by Kirchhoff \cite{K} in 1883 to describe the transversal oscillations of a
stretched string, particularly, taking into account the subsequent change in
string length caused by oscillations, where $u$ denotes the displacement, $f$
is the external force and $b$ is the initial tension while $a$ is related to
the intrinsic properties of the string, such as Young's modulus.

Since Lions \cite{L} introduced an abstract framework to the Kirchhoff type
equations, the qualitative analysis of nontrivial solutions for such
equations with various nonlinear terms, including the existence and
multiplicity of positive solutions and of sign-changing solution, has begun
to receive much attention; see, for example, \cite%
{ACM,BB,CKW,MR,MC,ML,N,S,SW3,ZP} for the bounded domain case and \cite%
{DPS,FIJ,HL,HZ,I,SW1,SW2} for the unbounded domain case.

Let us briefly comment some well-known results for the nonlinear Kirchhoff
type equations. Alves-Corr\^{e}a-Ma \cite{ACM} studied the following
Kirchhoff type equations in the bounded domain $\Omega :$

\begin{equation}
\left\{
\begin{array}{ll}
-M\left( \int_{\mathbb{R}^{N}}|\nabla u|^{2}dx\right) \Delta u=f(x,u) &
\text{in }\Omega , \\
u\in H_{0}^{1}(\Omega ). &
\end{array}%
\right.  \label{2}
\end{equation}%
They concluded the existence of positive solutions of Eq. (\ref{2}) when $M$
does not grow too fast in a suitable interval near zero and $f$ is locally
Lipschitz subject to some prescribed criteria. Bensedik-Bouchekif \cite{BB}
studied the asymptotically linear case and obtained the existence of
positive solutions of Eq. (\ref{2}) when the function $M$ is a
non-decreasing function and $M(t)\geq m_{0}$ for some $m_{0}>0$, and $f$ is
the asymptotically linear satisfying some assumptions about its asymptotic
behaviors near zero and infinite. Later, Chen-Kuo-Wu \cite{CKW} illustrated
the difference in the solution behavior which arises from the consideration
of the nonlocal effect for Eq. (\ref{2}) with $M(t)=at+b$ and $f$ being
concave-convex nonlinearity.

Compared with the bounded domain case, the unbounded domain case seems to be
more delicate. The primary difficulty lies in the lack of the embedding of
compactness. Figueiredo-Ikoma-J\'{u}nior \cite{FIJ} studied the existence
and concentration behaviors of positive solutions to the following Kirchhoff
type equations:

\begin{equation}
\left\{
\begin{array}{ll}
-\varepsilon ^{2}M\left( \varepsilon ^{2-N}\int_{\mathbb{R}^{N}}|\nabla
u|^{2}dx\right) \Delta u+V\left( x\right) u=f(u) & \text{in }\mathbb{R}^{N},
\\
u\in H^{1}(\mathbb{R}^{N}). &
\end{array}%
\right.  \label{3}
\end{equation}%
Under suitable conditions on $M$ and $f$, a family of positive solutions for
Eq. (\ref{3}) concentrating at a local minimum of $V$ are constructed.
Recently, we \cite{SW1} introduced the hypotheses $(V1)-(V2)$ to Kirchhoff
type equations in $\mathbb{R}^{N}$($N\geq 3$) and studied the existence of
nontrivial solution for Eq. $\left( K_{\lambda ,a}\right) $ with $m(t)=t$
and $f$ being asymptotically linear or superlinear at infinity on $u.$

Inspired by the above facts, in this paper we are likewise interested in
looking for nontrivial solutions for Eq. $(K_{a,\lambda })$ in $\mathbb{R}%
^{N}$ with $N\geq 1$. However, distinguishing from the existing literatures,
we are more focus on the interaction between the functions $m$ and $f,$
leading to the difference in the number of solutions. Specifically, we find
that the powers of $m$ and $f$ will dominate the number of solutions for Eq.
$(K_{a,\lambda }).$ We require that the function $m$ satisfies some
asymptotic behaviors near infinite and that $f$ is $k$-asymptotically linear
at infinity on $u$ for any real number $2<k<2^{\ast },$ i.e., $%
\lim_{|u|\rightarrow \infty }f\left( x,u\right) /|u|^{k-1}=q\left( x\right) $
uniformly in $x\in \mathbb{R}^{N},\ $while not requiring any assumption
about the asymptotic behavior near zero of $m$ and $f.$

We wish to point out that in the study of one positive solution, the range
of the parameter $a>0$ in Eq. $(K_{a,\lambda })$ is dependent on the
limiting function $q$ of $f.$ In other words, the different types of $q$
will bring about the different ranges of $a.$ Moreover, in the study of two
positive solutions, the geometry of the variational structure of Eq. $%
(K_{a,\lambda })$ is known to have a local minimum and a mountain pass,
since the power of $m$ is greater than the one of $f.$ In view of this, it
is clear to use the minimax method to seek two solutions of Eq. $%
(K_{a,\lambda })$ as critical points of the associated energy functional $%
J_{a,\lambda }$. However, since the norms $\Vert u\Vert _{D^{1,2}}=\left(
\int_{\mathbb{R}^{N}}|\nabla u|^{2}dx\right) ^{1/2}$ and $\Vert u\Vert
_{H^{1}}=\left( \int_{\mathbb{R}^{N}}(|\nabla u|^{2}+u^{2})dx\right) ^{1/2}$
are not equivalent in $H^{1}(\mathbb{R}^{N})$, we can not apply the standard
techniques to verify the boundedness below of $J_{a,\lambda }$ and the
boundedness of the $(PS)$-sequence.

Based on the analysis above, we suggest some new techniques and introduce
new hypotheses on $m$ and $q$ in the present paper. By using the minimax
method and Caffarelli-Kohn-Nirenberg inequality, we obtain the existence and
multiplicity of positive solutions for Eq. $(K_{a,\lambda })$ under the
different assumptions on $m$ and $f$, respectively.

We now summarize our main results as follows.

\begin{theorem}
\label{t1}Suppose that $N\geq 1$ and conditions $(V1)-(V2)$ hold. In
addition, for any real number $2<k<2^{\ast },$ we assume that the functions $%
m$ and $f$ satisfy the following conditions:

\begin{itemize}
\item[$\left( L1\right) $] there exists $m_{\infty }>0$ such that $%
\lim_{t\rightarrow \infty }t^{-(k-2)/2}m(t)=m_{\infty }$ and%
\begin{equation*}
\int_{\sigma }^{\eta }m(t)dt\geq \frac{2(\eta -\sigma )}{k}m(\eta )
\end{equation*}%
for all $0\leq \sigma <\eta ;$

\item[$\left( D1\right) $] there exist $q\in L^{\infty }(\mathbb{R}^{N})$
and $0\leq \mu <N-\frac{k(N-2)}{2}$ satisfying $q\not\equiv 0$ on $\overline{%
\Omega }$ and $\liminf_{\left\vert x\right\vert \rightarrow \infty
}\left\vert x\right\vert ^{\mu }q\left( x\right) >0$ such that%
\begin{equation*}
\lim_{s\rightarrow \infty }\frac{f\left( x,s\right) }{s^{k-1}}=q\left(
x\right) \text{ uniformly in }x\in \mathbb{R}^{N};
\end{equation*}

\item[$\left( D2\right) $] $s\mapsto \frac{f(x,s)}{s^{k-1}}$ is
nondecreasing function on $(0,\infty )$ for any fixed $x\in \mathbb{R}.$

Then there exists $\widetilde{{\Lambda }}>0$ such that Eq. $\left(
K_{\lambda ,a}\right) $ admits at least one positive solution for all $%
\lambda >\widetilde{{\Lambda }}$ and $a>0.$
\end{itemize}
\end{theorem}

\begin{remark}
\label{r1}$(i)$ It is not difficult to find such functions $m$ satisfying
condition $(L1)$. For example, let $m(t)=t^{(k-2)/2}+\theta t^{(k-2)/4}$ for
$t>0$, where $2<k<2^{\ast }$ and $\theta \geq 0.$ Clearly, $%
\lim_{t\rightarrow \infty }t^{-(k-2)/2}m(t)=1$. Moreover, a direct
calculation shows that%
\begin{eqnarray*}
\int_{\sigma }^{\eta }m(t)dt &=&\frac{2}{k}\left( \eta ^{k/2}-\sigma
^{k/2}\right) +\frac{4\theta }{k+2}\left( \eta ^{\left( k+2\right)
/4}-\sigma ^{\left( k+2\right) /4}\right)  \\
&\geq &\frac{2}{k}\left( \eta ^{k/2}-\eta ^{\left( k-2\right) /2}\sigma
\right) +\frac{4\theta }{k+2}\left( \eta ^{\left( k+2\right) /4}-\eta
^{\left( k-2\right) /4}\sigma \right)  \\
&>&\frac{2(\eta -\sigma )}{k}m\left( \eta \right) \text{ for all }0\leq
\sigma <\eta .
\end{eqnarray*}%
$(ii)$ Under condition $(D1),$ it is not difficult to verify that for any
real number $2<k<2^{\ast },$%
\begin{equation*}
\inf_{u\in H^{1}(\mathbb{R}^{N})}\frac{\left( \int_{\mathbb{R}^{N}}|\nabla
u|^{2}dx\right) ^{k/2}}{\int_{\mathbb{R}^{N}}q(x)|u|^{k}dx}=0.
\end{equation*}%
For more details, we refer to the proof of Lemma \ref{lem3} below.
\end{remark}

We now assume that the function $m$ satisfies the following assumptions
instead of condition $(L1)$:

\begin{itemize}
\item[$\left( L2\right) $] $m(t)$ is nondecreasing on $t\geq 0;$

\item[$(L3)$] there exist three positive numbers $m_{0},\delta $ and $T_{0}$
such that $m(t)\geq m_{0}t^{\delta }$ for all $t\geq T_{0}.$
\end{itemize}

Then we have the following result.

\begin{theorem}
\label{t2}Suppose that $N\geq 3,$ conditions $(V1)-(V2),(L2)$ and $(L3)$
with $\delta \geq \frac{2}{N-2}$ hold. In addition, for any real number $%
2<k<2^{\ast },$ we assume that the function $f$ satisfies conditions $%
(D1)-(D2).$ Then there exist constants $\widetilde{a}_{\ast },\widetilde{%
\Lambda }_{\ast }>0$ such that for every $0<a<\widetilde{a}_{\ast }$ and $%
\lambda >\widetilde{\Lambda }_{\ast },$ Eq. $(K_{a,\lambda })$ admits at
least two positive solutions $u_{a,\lambda }^{-}$ and $u_{a,\lambda }^{+}$
satisfying $J_{a,\lambda }(u_{a,\lambda }^{-})<0<J_{a,\lambda }(u_{a,\lambda
}^{+}).$ In particular, $u_{a,\lambda }^{-}$ is a ground state solution of
Eq. $(K_{a,\lambda }).$
\end{theorem}

\begin{remark}
\label{r2} Compared with Theorem \ref{t1}, if we raise the power $\delta $
of the function $m$ such that $\delta \geq \frac{2}{N-2}>\frac{k-2}{2}$ for $%
N\geq 3$ in Theorem \ref{t2}, then two positive solutions of Eq. $%
(K_{a,\lambda })$ can be obtained.
\end{remark}

It is well known that for $N\geq 3$ and $2<k<2^{\ast },$ the following
minimum problem%
\begin{equation}
\overline{\nu }_{1}^{(k)}:=\inf_{u\in D^{1,2}(\mathbb{R}^{N})}\frac{\left(
\int_{\mathbb{R}^{N}}|\nabla u|^{2}dx\right) ^{k/2}}{\int_{\mathbb{R}%
^{N}}|x|^{\frac{k(N-2)}{2}-N}|u|^{k}dx}>0  \label{1-1}
\end{equation}%
is achieved by some $\overline{\varphi }_{k}\in D^{1,2}(\mathbb{R}^{N})$ by
Caffarelli-Kohn-Nirenberg inequality.

We now assume that the following assumption holds:

\begin{itemize}
\item[$\left( D1\right) ^{\prime }$] there exist the function $q(x)$
satisfying $q\not\equiv 0$ on $\overline{\Omega }$ and $q\left( x\right)
\leq c_{\ast }|x|^{\frac{k(N-2)}{2}-N}$ for some $c_{\ast }>0$ and for all $%
x\in \mathbb{R}^{N}$ such that
\begin{equation*}
\lim_{s\rightarrow \infty }\frac{f\left( x,s\right) }{s^{k-1}}=q\left(
x\right) \text{ uniformly in }x\in \mathbb{R}^{N}.
\end{equation*}
\end{itemize}

Under condition $\left( D1\right) ^{\prime }$ and $(\ref{1-1})$, it is
easily seen that for any $2<k<2^{\ast },$ the minimum problem%
\begin{equation}
\overline{\mu }_{1}^{(k)}:=\inf_{u\in H^{1}(\mathbb{R}^{N})}\frac{\left(
\int_{\mathbb{R}^{N}}|\nabla u|^{2}dx\right) ^{k/2}}{\int_{\mathbb{R}%
^{N}}q(x)|u|^{k}dx}\geq \frac{\overline{\nu }_{1}^{(k)}}{c_{\ast }}>0.
\label{1-2}
\end{equation}%
Then we have the following results.

\begin{theorem}
\label{t3}Suppose that $N\geq 3$ and conditions $(V1)-(V2)$ hold. In
addition, for any real number $2<k<2^{\ast },$ we assume that conditions $%
(L1),(D1)^{\prime }$ and $(D2)$ hold. Then for each $0<a<\frac{1}{m_{\infty }%
\overline{\mu }_{1}^{(k)}}$ there exists $\overline{\Lambda }>0$ such that
Eq. $\left( K_{\lambda ,a}\right) $ admits at least one positive solution
for all $\lambda >\overline{\Lambda }.$
\end{theorem}

\begin{theorem}
\label{t4}Suppose that $N\geq 3$ and conditions $(V1)-(V2),(L2)$ hold. In
addition, for any real number $2<k<2^{\ast },$ we assume that conditions $%
(L3)$ with $\delta >\frac{k-2}{2},\left( D1\right) ^{\prime }$ and $(D2)$
hold. Then there exists constants $\overline{a}_{\ast },\overline{\Lambda }%
_{\ast }>0$ such that for every $0<a<\overline{a}_{\ast }$ and $\lambda >%
\overline{\Lambda }_{\ast },$ Eq. $(K_{a,\lambda })$ admits at least two
positive solutions $u_{a,\lambda }^{-}$ and $u_{a,\lambda }^{+}$ satisfying $%
J_{a,\lambda }(u_{a,\lambda }^{-})<0<J_{a,\lambda }(u_{a,\lambda }^{+})$ $.$
In particular, $u_{a,\lambda }^{-}$ is a ground state solution of Eq. $%
(K_{a,\lambda }).$
\end{theorem}

\begin{remark}
\label{ex3} Unlike Theorem \ref{t2}, in Theorem \ref{t4} we only require the
power $\delta $ of $m(t)$ is greater than $\frac{k-2}{2},$ also leading to
two positive solutions of Eq. $(K_{a,\lambda })$ under condition $\left(
D1\right) ^{\prime }.$
\end{remark}

The remainder of this paper is organized as follows. After giving some
preliminaries in Section 2, we prove that $J_{\lambda ,a}$ satisfies the
mountain pass geometry in Section 3. In Sections 4--6, we give the proofs of
Theorems \ref{t1}--\ref{t4}.

\section{Preliminaries}

Let%
\begin{equation*}
X=\left\{ u\in H^{1}(\mathbb{R}^{N})\ |\ \int_{\mathbb{R}^{N}}V\left(
x\right) u^{2}dx<\infty \right\}
\end{equation*}%
be equipped with the inner product and norm%
\begin{equation*}
\left\langle u,v\right\rangle =\int_{\mathbb{R}^{N}}\left( \nabla u\nabla
v+V(x)uv\right) dx,\ \left\Vert u\right\Vert =\left\langle u,u\right\rangle
^{1/2}.
\end{equation*}%
For $\lambda >0,$ we also need the following inner product and norm%
\begin{equation*}
\left\langle u,v\right\rangle _{\lambda }=\int_{\mathbb{R}^{N}}\left( \nabla
u\nabla v+\lambda V(x)uv\right) dx,\ \left\Vert u\right\Vert _{\lambda
}=\left\langle u,u\right\rangle _{\lambda }^{1/2}.
\end{equation*}%
Clearly, $\left\Vert u\right\Vert \leq \left\Vert u\right\Vert _{\lambda }$
for $\lambda \geq 1.$ Now we set $X_{\lambda }=( X,\left\Vert u\right\Vert
_{\lambda }).$

For $N=1,2,$ applying conditions $(V1)-(V2)$, the H\"{o}lder, Young and
Gagliardo-Nirenberg inequalities, there exists a sharp constant $A_{N}>0$
such that%
\begin{eqnarray*}
\int_{\mathbb{R}^{N}}u^{2}dx &\leq &\frac{1}{c_{0}}\int_{\left\{ V\geq
c_{0}\right\} }V\left( x\right) u^{2}dx+\left( \left\vert \left\{
V<c_{0}\right\} \right\vert \int_{\mathbb{R}^{N}}u^{4}dx\right) ^{1/2} \\
&\leq &\frac{1}{c_{0}}\int_{\mathbb{R}^{N}}V\left( x\right)
u^{2}dx+A_{N}^{2}\left\vert \left\{ V<c_{0}\right\} \right\vert
^{1/2}\left\Vert u\right\Vert _{D^{1,2}}^{N/2}\left\Vert u\right\Vert
_{L^{2}}^{\left( 4-N\right) /2} \\
&\leq &\frac{1}{c_{0}}\int_{\mathbb{R}^{N}}V\left( x\right) u^{2}dx+\frac{%
NA_{N}^{8/N}\left\vert \left\{ V<c_{0}\right\} \right\vert ^{2/N}\left\Vert
u\right\Vert _{D^{1,2}}^{2}}{4}+\frac{\left( 4-N\right) \left\Vert
u\right\Vert _{L^{2}}^{2}}{4},
\end{eqnarray*}%
which shows that%
\begin{equation*}
\int_{\mathbb{R}^{N}}u^{2}dx\leq \frac{4}{Nc_{0}}\int_{\mathbb{R}%
^{N}}V\left( x\right) u^{2}dx+A_{N}^{8/N}\left\vert \left\{ V<c_{0}\right\}
\right\vert ^{2/N}\left\Vert u\right\Vert _{D^{1,2}}^{2}.
\end{equation*}%
This implies that%
\begin{equation}
\left\Vert u\right\Vert _{H^{1}}^{2}\leq \max \left\{
1+A_{N}^{8/N}\left\vert \left\{ V<c_{0}\right\} \right\vert ^{2/N},\frac{4}{%
Nc_{0}}\right\} \left\Vert u\right\Vert ^{2}.  \label{39}
\end{equation}%
Similarly, we also have
\begin{equation}
\left\Vert u\right\Vert _{H^{1}}^{2}\leq \left( 1+A_{N}^{8/N}\left\vert
\left\{ V<c_{0}\right\} \right\vert ^{2/N}\right) \left\Vert u\right\Vert
_{\lambda }^{2}\text{ for }\lambda \geq \frac{4}{Nc_{0}}\left(
1+A_{N}^{8/N}\left\vert \left\{ V<c_{0}\right\} \right\vert ^{2/N}\right)
^{-1}.  \label{44}
\end{equation}%
For $N\geq 3,$ following \cite{SW1}, we have
\begin{equation}
\left\Vert u\right\Vert _{H^{1}}^{2}\leq \max \left\{ 1+S^{2}\left\vert
\left\{ V<c_{0}\right\} \right\vert ^{2/N},c_{0}^{-1}\right\} \left\Vert
u\right\Vert ^{2}  \label{41}
\end{equation}%
and
\begin{equation*}
\left\Vert u\right\Vert _{H^{1}}^{2}\leq \left( 1+S^{2}\left\vert \left\{
V<c_{0}\right\} \right\vert ^{2/N}\right) \left\Vert u\right\Vert _{\lambda
}^{2}\text{ for }\lambda \geq c_{0}^{-1}\left( 1+S^{2}\left\vert \left\{
V<c_{0}\right\} \right\vert ^{2/N}\right) ,  \label{45}
\end{equation*}%
where $S$ is the best constant for the embedding of $D^{1,2}(\mathbb{R}^{N})$
in $L^{2^{\ast }}(\mathbb{R}^{N})$. Set%
\begin{equation*}
\alpha _{N}:=\left\{
\begin{array}{ll}
\max \left\{ 1+A_{N}^{8/N}\left\vert \left\{ V<c_{0}\right\} \right\vert
^{2/N},\frac{4}{Nc_{0}}\right\}  & \text{ for }N=1,2; \\
\max \left\{ 1+S^{2}\left\vert \left\{ V<c_{0}\right\} \right\vert
^{2/N},c_{0}^{-1}\right\}  & \text{ for }N\geq 3.%
\end{array}%
\right.
\end{equation*}%
Thus, it follows from (\ref{39}) and (\ref{41}) that
\begin{equation*}
\left\Vert u\right\Vert _{H^{1}}^{2}\leq \alpha _{N}\left\Vert u\right\Vert
^{2},  \label{9}
\end{equation*}%
which implies that the imbedding $X\hookrightarrow H^{1}(\mathbb{R}^{N})$ is
continuous.

Since the imbedding $H^{1}(\mathbb{R}^{N})\hookrightarrow L^{r}(\mathbb{R}%
^{N})$ $(2\leq r<+\infty )$ is continuous for $N=1,2$, from $(\ref{44})$ it
follows that for any $r\in \lbrack 2,+\infty )$,%
\begin{equation*}
\int_{\mathbb{R}^{N}}\left\vert u\right\vert ^{r}dx\leq S_{r}^{-r}\left\Vert
u\right\Vert _{H^{1}}^{r}\leq S_{r}^{-r}\left( 1+A_{N}^{8/N}\left\vert
\left\{ V<c_{0}\right\} \right\vert ^{2/N}\right) ^{r/2}\left\Vert
u\right\Vert _{\lambda }^{r}
\end{equation*}%
for $\lambda \geq \frac{4}{Nc_{0}}\left( 1+A_{N}^{8/N}\left\vert \left\{
V<c_{0}\right\} \right\vert ^{2/N}\right) ^{-1}$, where $S_{r}$ is the best
Sobolev constant for the imbedding of $H^{1}(\mathbb{R}^{N})$ in $L^{r}(%
\mathbb{R}^{N})$ ($2\leq r<+\infty $). For $N\geq 3,$ following the argument
in \cite{SW1} (see pp 1776-1777), for any $r\in \lbrack 2,2^{\ast })$ one has%
\begin{eqnarray}
&&\int_{\mathbb{R}^{N}}\left\vert u\right\vert ^{r}dx  \notag \\
&\leq &\left( \int_{\left\{ V\geq c_{0}\right\} }u^{2}dx+\int_{\left\{
V<c_{0}\right\} }u^{2}dx\right) ^{\frac{2^{\ast }-r}{2^{\ast }-2}}\left(
S^{-1}\left\Vert u\right\Vert _{D^{1,2}}\right) ^{\frac{2^{\ast }(r-2)}{%
2^{\ast }-2}}  \notag \\
&\leq &\left( \frac{1}{\lambda c_{0}}\int_{\mathbb{R}^{N}}\lambda
V(x)u^{2}dx+S^{-2}\left\vert \left\{ V<c_{0}\right\} \right\vert
^{2/N}\left\Vert u\right\Vert _{D^{1,2}}^{2}\right) ^{\frac{2^{\ast }-r}{%
2^{\ast }-2}}\left( S^{-1}\left\Vert u\right\Vert _{D^{1,2}}\right) ^{\frac{%
2^{\ast }(r-2)}{2^{\ast }-2}}  \label{10} \\
&\leq &\left\vert \left\{ V<c_{0}\right\} \right\vert ^{\left( 2^{\ast
}-r\right) /2^{\ast }}S^{-r}\left\Vert u\right\Vert _{\lambda }^{r}\text{
for }\lambda \geq c_{0}^{-1}S^{2}\left\vert \left\{ V<c_{0}\right\}
\right\vert ^{-2/N}.  \notag
\end{eqnarray}%
Now we set%
\begin{equation}
\Theta _{r,N}:=\left\{
\begin{array}{ll}
S_{r}^{-r}\left( 1+A_{N}^{8/N}\left\vert \left\{ V<c_{0}\right\} \right\vert
^{2/N}\right) ^{r/2} & \text{ if }N=1,2; \\
S^{-r}\left\vert \left\{ V<c_{0}\right\} \right\vert ^{\frac{2^{\ast }-r}{%
2^{\ast }}} & \text{ if }N\geq 3,%
\end{array}%
\right.   \label{7}
\end{equation}%
and
\begin{equation}
\Lambda _{N}:=\left\{
\begin{array}{ll}
\frac{4}{Nc_{0}}\left( 1+A_{N}^{8/N}\left\vert \left\{ V<c_{0}\right\}
\right\vert ^{2/N}\right) ^{-1} & \text{ if }N=1,2; \\
c_{0}^{-1}S^{2}\left\vert \left\{ V<c_{0}\right\} \right\vert ^{-2/N} &
\text{ if }N\geq 3.%
\end{array}%
\right.   \label{8}
\end{equation}%
Thus, by (\ref{7})--(\ref{8}) we have for any $r\in \lbrack 2,2^{\ast })$
and $\lambda \geq \Lambda _{N},$
\begin{equation}
\int_{\mathbb{R}^{N}}\left\vert u\right\vert ^{r}dx\leq \Theta
_{r,N}\left\Vert u\right\Vert _{\lambda }^{r}.  \label{11}
\end{equation}

Eq. $(K_{\lambda ,a})$ is variational and its solutions are the critical
points of the functional defined in $X$ by
\begin{equation}
J_{\lambda ,a}\left( u\right) =\frac{a}{2}\widehat{m}\left( \left\Vert
u\right\Vert _{D^{1,2}}^{2}\right) +\frac{1}{2}\left( b\left\Vert
u\right\Vert _{D^{1,2}}^{2}+\int_{\mathbb{R}^{N}}\lambda V(x)u^{2}dx\right)
-\int_{\mathbb{R}^{N}}F(x,u)dx,  \label{12}
\end{equation}%
where $\widehat{m}(t)=\int_{0}^{t}m(s)ds$ and $F(x,u)=\int_{0}^{u}f(x,s)ds$.
Furthermore, it is not difficult to prove that the functional $J_{\lambda
,a} $ is of class $C^{1}$ in $X$, and that
\begin{eqnarray}
\langle J_{\lambda ,a}^{\prime }(u),v\rangle &=&\left( am\left( \left\Vert
u\right\Vert _{D^{1,2}}^{2}\right) +b\right) \int_{\mathbb{R}^{N}}\nabla
u\nabla vdx+\int_{\mathbb{R}^{N}}\lambda V(x)uvdx  \notag \\
&&-\int_{\mathbb{R}^{N}}f(x,u)vdx.  \label{13}
\end{eqnarray}

The following theorem is a variant version of the mountain pass theorem,
which helps us find a so-called Cerami type $(PS)$-sequence.

\begin{theorem}
\label{t7}(\cite{E2}, Mountain Pass Theorem) Let $E$ be a real Banach space
with its dual space $E^{\ast },$ and suppose that $I\in C^{1}(E,\mathbb{R})$
satisfies
\begin{equation*}
\max \{I(0),I(e)\}\leq \mu <\eta \leq \inf_{\Vert u\Vert =\rho }I(u),
\end{equation*}%
for some $\mu <\eta ,\rho >0$ and $e\in E$ with $\Vert e\Vert >\rho .$ Let $%
\alpha \geq \eta $ be characterized by
\begin{equation*}
\alpha =\inf_{\gamma \in \Gamma }\max_{0\leq \tau \leq 1}I(\gamma (\tau )),
\end{equation*}%
where $\Gamma =\{\gamma \in C([0,1],E):\gamma (0)=0,\gamma (1)=e\}$ is the
set of continuous paths joining $0$ and $e$. Then there exists a sequence $%
\{u_{n}\}\subset E$ such that
\begin{equation*}
I(u_{n})\rightarrow \alpha \geq \eta \quad \text{and}\quad (1+\Vert
u_{n}\Vert )\Vert I^{\prime }(u_{n})\Vert _{E^{\ast }}\rightarrow 0,\quad
\text{as}\ n\rightarrow \infty .
\end{equation*}
\end{theorem}

\section{Mountain pass geometry}

In this section we prove that the energy functional $J_{\lambda ,a}$
satisfies the mountain pass geometry under the different assumptions on $m$
and $f$, respectively.

\begin{lemma}
\label{lem1}Suppose that conditions $(V1)-(V2)$ and $(D1)-(D2)$ hold. Then
there exist $\rho >0$ and $\eta >0$ such that $\inf \{J_{\lambda ,a}(u):u\in
X_{\lambda }\ $with$\ \left\Vert u\right\Vert _{\lambda }=\rho \}>\eta $ for
all $\lambda \geq \Lambda _{N}.$
\end{lemma}

\begin{proof}
By conditions $(D1)-(D2)$, we obtain that%
\begin{equation}
f(x,s)\leq q(x)s^{k-1}\quad \text{for all}\ s\geq 0  \label{3.1}
\end{equation}%
and
\begin{equation}
F(x,s)\leq \frac{1}{k}q(x)s^{k}\quad \text{for all}\ s\geq 0.  \label{3.2}
\end{equation}%
Then, by (\ref{11}) and (\ref{3.2}), for every $u\in X$ and $\lambda \geq
\Lambda _{N}$ one has
\begin{equation*}
\int_{\mathbb{R}^{N}}F(x,u)dx\leq \frac{|q|_{\infty }}{k}\int_{\mathbb{R}%
^{N}}|u|^{k}dx\leq \frac{|q|_{\infty }\Theta _{k,N}}{k}\Vert u\Vert
_{\lambda }^{k}.
\end{equation*}%
This implies that
\begin{eqnarray*}
J_{\lambda ,a}(u) &=&\frac{a}{2}\widehat{m}\left( \left\Vert u\right\Vert
_{D^{1,2}}^{2}\right) +\frac{1}{2}\left( b\int_{\mathbb{R}^{N}}\left\vert
\nabla u\right\vert ^{2}dx+\int_{\mathbb{R}^{N}}\lambda V(x)u^{2}dx\right)
-\int_{\mathbb{R}^{N}}F(x,u)dx \\
&\geq &\frac{1}{2}\min \left\{ b,1\right\} \Vert u\Vert _{\lambda }^{2}-%
\frac{|q|_{\infty }\Theta _{k,N}}{k}\Vert u\Vert _{\lambda }^{k}.
\end{eqnarray*}%
Thus, letting $\left\Vert u\right\Vert _{\lambda }=\rho >0$ small enough, it
is easy to obtain that there exists $\eta >0$ such that $\inf \{J_{\lambda
,a}(u):u\in X_{\lambda }\ $with$\ \left\Vert u\right\Vert _{\lambda }=\rho
\}>\eta $ for all $\lambda \geq \Lambda _{N},$ since $2<k<2^{\ast }.$ The
lemma is proved.
\end{proof}

\begin{lemma}
\label{lem2}Suppose that conditions $(V1)-(V2),(D1)^{\prime }$ and $\left(
D2\right) $ hold. Then there exist $\rho >0$ and $\eta >0$ such that $\inf
\{J_{\lambda ,a}(u):u\in X_{\lambda }\ $with$\ \left\Vert u\right\Vert
_{\lambda }=\rho \}>\eta $ for all $a>0$ and $\lambda \geq \Lambda _{N}.$
\end{lemma}

\begin{proof}
It follows from conditions $\left( D1\right) ^{\prime }$ and $\left(
D2\right) $ that%
\begin{equation}
f(x,s)\leq c_{\ast }\left\vert x\right\vert ^{\frac{k(N-2)}{2}%
-N}s^{k-1}\quad \text{for all}\ s\geq 0  \label{3.6}
\end{equation}%
and
\begin{equation}
F(x,s)\leq \frac{c_{\ast }}{k}\left\vert x\right\vert ^{\frac{k(N-2)}{2}%
-N}s^{k}\quad \text{for all}\ s\geq 0.  \label{3.3}
\end{equation}%
Then, by (\ref{1-1}) and (\ref{3.3}), for every $u\in X$ and $\lambda \geq
\Lambda _{N}$ one has
\begin{eqnarray}
\int_{\mathbb{R}^{N}}F(x,u)dx &\leq &\frac{c_{\ast }}{k}\int_{\mathbb{R}%
^{N}}\left\vert x\right\vert ^{\frac{k(N-2)}{2}-N}|u|^{k}dx\leq \frac{%
c_{\ast }}{k\overline{\nu }_{1}^{(k)}}\left( \int_{\mathbb{R}^{N}}|\nabla
u|^{2}dx\right) ^{k/2}  \label{3.4} \\
&\leq &\frac{c_{\ast }}{k\overline{\nu }_{1}^{(k)}}\Vert u\Vert _{\lambda
}^{k},  \notag
\end{eqnarray}%
which implies that
\begin{eqnarray*}
J_{\lambda ,a}(u) &\geq &\frac{1}{2}\left( b\int_{\mathbb{R}^{N}}\left\vert
\nabla u\right\vert ^{2}dx+\int_{\mathbb{R}^{N}}\lambda V(x)u^{2}dx\right) -%
\frac{c_{\ast }}{k\overline{\nu }_{1}^{(k)}}\Vert u\Vert _{\lambda }^{k} \\
&\geq &\frac{1}{2}\min \left\{ b,1\right\} \Vert u\Vert _{\lambda }^{2}-%
\frac{c_{\ast }}{k\overline{\nu }_{1}^{(k)}}\Vert u\Vert _{\lambda }^{k}.
\end{eqnarray*}%
Thus, letting $\left\Vert u\right\Vert _{\lambda }=\rho >0$ small enough, it
is easy to obtain that there exists $\eta >0$ such that $\inf \{J_{\lambda
,a}(u):u\in X_{\lambda }\ $with$\ \left\Vert u\right\Vert _{\lambda }=\rho
\}>\eta $ for all $\lambda \geq \Lambda _{N},$ since $2<k<2^{\ast }.$ The
lemma is proved.
\end{proof}

\begin{lemma}
\label{lem3}Suppose that conditions $(V1)-(V2),(L1)$ and $(D1)-(D2)$ hold.
Let $\rho >0$ be as Lemma \ref{lem1}. Then there exists $e\in X$ with $\Vert
e\Vert _{\lambda }>\rho $ such that $J_{\lambda ,a}(e)<0$ for all $a>0$ and $%
\lambda >0.$
\end{lemma}

\begin{proof}
Let $u\in X\backslash \left\{ 0\right\} $ with $u>0$ and define $u_{n}\left(
x\right) =n^{-N/k}u\left( \frac{x}{n}\right) .$ A direct calculation shows
that%
\begin{equation*}
\int_{\mathbb{R}^{N}}\left\vert \nabla u_{n}\right\vert ^{2}dx=n^{N-2-\frac{%
2N}{k}}\int_{\mathbb{R}^{N}}\left\vert \nabla u\right\vert ^{2}dx
\end{equation*}%
and%
\begin{equation*}
\int_{\mathbb{R}^{N}}q(x)u_{n}^{k}dx=n^{-N}\int_{\mathbb{R}%
^{N}}q(x)u^{k}\left( \frac{x}{n}\right) dx=\int_{\mathbb{R}%
^{N}}q(nx)u^{k}\left( x\right) dx.
\end{equation*}%
Thus, it follows from condition $(D1)$ and Fatou's lemma that%
\begin{eqnarray*}
\frac{\left\Vert u_{n}\right\Vert _{D^{1,2}}^{k}}{\int_{\mathbb{R}%
^{N}}q(x)u_{n}^{k}dx} &=&\frac{n^{-\left( N-\frac{k(N-2)}{2}\right)
}\left\Vert u\right\Vert _{D^{1,2}}^{k}}{\int_{\mathbb{R}^{N}}q(nx)u^{k}%
\left( x\right) dx}=\frac{\left\Vert u\right\Vert _{D^{1,2}}^{k}}{n^{N-\frac{%
k(N-2)}{2}}\int_{\mathbb{R}^{N}}q(nx)u^{k}\left( x\right) dx} \\
&\leq &\frac{R_{0}^{\mu }\left\Vert u\right\Vert _{D^{1,2}}^{k}}{n^{N-\frac{%
k(N-2)}{2}-\mu }\int_{\left\vert x\right\vert \leq R_{0}}\left\vert
nx\right\vert ^{\mu }q(nx)u^{k}\left( x\right) dx} \\
&\leq &\frac{R_{0}^{\mu }\left\Vert u\right\Vert _{D^{1,2}}^{k}}{Cn^{N-\frac{%
k(N-2)}{2}-\mu }\int_{\left\vert x\right\vert \leq R_{0}}u^{k}dx}\rightarrow
0\text{ as }n\rightarrow \infty ,
\end{eqnarray*}%
which indicates that%
\begin{equation*}
\inf_{u\in X}\frac{\left( \int_{\mathbb{R}^{N}}\left\vert \nabla
u\right\vert ^{2}dx\right) ^{k/2}}{\int_{\mathbb{R}^{N}}q(x)|u|^{k}dx}=0.
\label{3-7}
\end{equation*}%
Thus, for each $a>0$, there exists $\phi _{k}\in X\backslash \left\{
0\right\} $ with $\phi _{k}>0$ such that%
\begin{equation*}
am_{\infty }\left\Vert \phi _{k}\right\Vert _{D^{1,2}}^{k}-\int_{\mathbb{R}%
^{N}}q(x)\phi _{k}^{k}dx<0.
\end{equation*}%
Using the above inequality, together with conditions $(L1),(D1)-(D2)$ and
Lebesgue's dominated convergence theorem, leads to
\begin{eqnarray*}
\lim_{t\rightarrow +\infty }\frac{J_{\lambda ,a}(t\phi _{k})}{t^{k}}
&=&\lim_{t\rightarrow +\infty }\frac{1}{2t^{k-2}}\left( b\left\Vert \phi
_{k}\right\Vert _{D^{1,2}}^{2}+\int_{\mathbb{R}^{N}}\lambda V(x)\phi
_{k}^{2}dx\right)  \\
&&+\lim_{t\rightarrow +\infty }\left[ \frac{a\hat{m}\left( t^{2}\left\Vert
\phi _{k}\right\Vert _{D^{1,2}}^{2}\right) }{2t^{k}\left\Vert \phi
_{k}\right\Vert _{D^{1,2}}^{k}}\left\Vert \phi _{k}\right\Vert
_{D^{1,2}}^{k}-\int_{\mathbb{R}^{N}}\frac{F(x,t\phi _{k})}{t^{k}\phi _{k}^{k}%
}\phi _{k}^{k}dx\right]  \\
&\leq &\frac{am_{\infty }}{k}\left\Vert \phi _{k}\right\Vert _{D^{1,2}}^{k}-%
\frac{1}{k}\int_{\mathbb{R}^{N}}q(x)\phi _{k}^{k}dx \\
&=&\frac{1}{k}\left( am_{\infty }\left\Vert \phi _{k}\right\Vert
_{D^{1,2}}^{k}-\int_{\mathbb{R}^{N}}q(x)\phi _{k}^{k}dx\right) <0.
\end{eqnarray*}%
This implies that $J_{\lambda ,a}(t\phi _{k})\rightarrow -\infty $ as $%
t\rightarrow +\infty .$ Therefore, there exists $e\in X$ with $\Vert e\Vert
_{\lambda }>\rho $ such that $J_{\lambda ,a}(e)<0$ and the lemma is proved.
\end{proof}

Note that if condition $(L1)$ is removed, then we can also arrive at a
conclusion similar to Lemma \ref{lem3}, but the parameter $a>0$ must be
small. Now we state this result.

\begin{lemma}
\label{lem10}Suppose that conditions $(V1)-(V2)$ and $(D1)-(D2)$ hold. Let $%
\rho >0$ be as Lemma \ref{lem1}. Then there exist $\widetilde{a}_{\ast }>0$
and $e\in X$ with $\Vert e\Vert _{\lambda }>\rho $ such that $J_{\lambda
,a}(e)<0$ for all $0<a<\widetilde{a}_{\ast }$ and $\lambda >0.$
\end{lemma}

\begin{proof}
According to the argument of Lemma \ref{lem3}, there exists $\phi _{k}\in
X\backslash \left\{ 0\right\} $ with $\phi _{k}>0$ such that $\int_{\mathbb{R%
}^{N}}q(x)\phi _{k}^{k}dx>0$ by condition $(D1).$ Then using conditions $%
(D1)-(D2)$, together with Lebesgue's dominated convergence theorem one has
\begin{eqnarray*}
\lim_{t\rightarrow +\infty }\frac{J_{\lambda ,0}(t\phi _{k})}{t^{k}}
&=&\lim_{t\rightarrow +\infty }\frac{1}{2t^{k-2}}\left( b\left\Vert \phi
_{k}\right\Vert _{D^{1,2}}^{2}+\int_{\mathbb{R}^{N}}\lambda V(x)\phi
_{k}^{2}dx\right) -\lim_{t\rightarrow +\infty }\int_{\mathbb{R}^{N}}\frac{%
F(x,t\phi _{k})}{t^{k}\phi _{k}^{k}}\phi _{k}^{k}dx \\
&\leq &-\frac{1}{k}\int_{\mathbb{R}^{N}}q(x)\phi _{k}^{k}dx<0,
\end{eqnarray*}%
where $J_{\lambda ,0}(u)=J_{\lambda ,a}(u)$ with $a=0.$ Thus, if $J_{\lambda
,0}(t\phi _{k})\rightarrow -\infty $ as $t\rightarrow +\infty ,$ then there
exists $e\in X$ with $\Vert e\Vert _{\lambda }>\rho $ such that $J_{\lambda
,0}(e)<0.$ Since $J_{\lambda ,a}(e)\rightarrow J_{\lambda ,0}(e)$ as $%
a\rightarrow 0^{+},$ we obtain that there exists $\widetilde{a}_{\ast }>0$
such that $J_{\lambda ,a}(e)<0$ for all $0<a<\widetilde{a}_{\ast }$ and $%
\lambda >0.$
\end{proof}

\begin{remark}
\label{re1} We point out that the value of $\widetilde{a}_{\ast }$ can not
be determined in general, but only in some special cases. For example, let
us assume that $N=4,m\left( t\right) =t$ and%
\begin{equation*}
f\left( x,s\right) =\left\{
\begin{array}{ll}
q\left( x\right) s^{2} & \text{ if }s\geq 0, \\
0 & \text{ if }s<0,%
\end{array}%
\right.
\end{equation*}%
where the function $q$ is as in condition $\left( D1\right) .$ Clearly, the
function $m(t)$ does not satisfy condition $(L1).$ Define the minimum problem%
\begin{equation*}
\widehat{\mu }_{1}(\lambda )=\inf_{u\in H^{1}(\mathbb{R}^{N})}\frac{\left(
\int_{\mathbb{R}^{N}}|\nabla u|^{2}dx\right) ^{2}\left( b\int_{\mathbb{R}%
^{N}}|\nabla u|^{2}dx+\int_{\mathbb{R}^{N}}\lambda V(x)u^{2}dx\right) }{%
\left( \int_{\mathbb{R}^{N}}q\left( x\right) \left\vert u\right\vert
^{3}dx\right) ^{2}}.
\end{equation*}%
Then $\widehat{\mu }_{1}(\lambda )\geq \frac{\min \left\{ 1,b\right\} S^{6}}{%
|q|_{\infty }^{2}\left\vert \left\{ V<c_{0}\right\} \right\vert ^{1/2}}>0$
for all $\lambda >\frac{S^{2}}{c_{0}\left\vert \left\{ V<c_{0}\right\}
\right\vert ^{1/2}}.$ Indeed, by $\left( \ref{10}\right) ,$ for every $%
\lambda >\frac{S^{2}}{c_{0}\left\vert \left\{ V<c_{0}\right\} \right\vert
^{1/2}}$ there holds%
\begin{eqnarray*}
&&\frac{\left( \int_{\mathbb{R}^{N}}|\nabla u|^{2}dx\right) ^{2}\left(
b\int_{\mathbb{R}^{N}}|\nabla u|^{2}dx+\int_{\mathbb{R}^{N}}\lambda
V(x)u^{2}dx\right) }{\left( \int_{\mathbb{R}^{N}}q\left( x\right) \left\vert
u\right\vert ^{3}dx\right) ^{2}} \\
&\geq &\frac{\left( \int_{\mathbb{R}^{N}}|\nabla u|^{2}dx\right) ^{2}\left(
b\int_{\mathbb{R}^{N}}|\nabla u|^{2}dx+\int_{\mathbb{R}^{N}}\lambda
V(x)u^{2}dx\right) }{|q|_{\infty }^{2}\left( \int_{\mathbb{R}^{N}}\left\vert
u\right\vert ^{3}dx\right) ^{2}} \\
&\geq &\frac{\min \left\{ 1,b\right\} S^{6}\left\Vert u\right\Vert
_{D^{1,2}}^{4}\left\Vert u\right\Vert _{\lambda }^{2}}{|q|_{\infty
}^{2}\left\vert \left\{ V<c_{0}\right\} \right\vert ^{1/2}\left\Vert
u\right\Vert _{\lambda }^{2}\left\Vert u\right\Vert _{D^{1,2}}^{4}} \\
&=&\frac{\min \left\{ 1,b\right\} S^{6}}{|q|_{\infty }^{2}\left\vert \left\{
V<c_{0}\right\} \right\vert ^{1/2}}>0,
\end{eqnarray*}%
which implies that $\widehat{\mu }_{1}\geq \frac{\min \left\{ 1,b\right\}
S^{6}}{|q|_{\infty }^{2}\left\vert \left\{ V<c_{0}\right\} \right\vert ^{1/2}%
}>0.$ Thus, for every $0<a<\frac{2}{9\widehat{\mu }_{1}^{2}(\lambda )}$
there exists $u_{0}\in X_{\lambda }$ such that%
\begin{equation}
\frac{1}{9}\left( \int_{\mathbb{R}^{N}}q\left( x\right) \left\vert
u_{0}\right\vert ^{3}dx\right) ^{2}>\frac{a}{2}\left\Vert u_{0}\right\Vert
_{D^{1,2}}^{4}\left( b\left\Vert u_{0}\right\Vert _{D^{1,2}}^{2}+\int_{%
\mathbb{R}^{N}}\lambda V(x)u_{0}^{2}dx\right) .  \label{3-10}
\end{equation}%
Note that%
\begin{eqnarray*}
J_{\lambda ,a}(tu_{0}) &=&\frac{at^{4}}{4}\left\Vert u_{0}\right\Vert
_{D^{1,2}}^{4}+\frac{t^{2}}{2}\left( b\left\Vert u_{0}\right\Vert
_{D^{1,2}}^{2}+\int_{\mathbb{R}^{N}}\lambda V(x)u_{0}^{2}dx\right) -\int_{%
\mathbb{R}^{N}}F(x,tu_{0})dx \\
&=&t^{2}\left[ \frac{at^{2}}{4}\left\Vert u_{0}\right\Vert _{D^{1,2}}^{4}+%
\frac{1}{2}\left( b\left\Vert u_{0}\right\Vert _{D^{1,2}}^{2}+\int_{\mathbb{R%
}^{N}}\lambda V(x)u_{0}^{2}dx\right) -\frac{t}{3}\int_{\mathbb{R}%
^{N}}q\left( x\right) \left\vert u_{0}\right\vert ^{3}dx\right]  \\
&=&t^{2}g(tu_{0}),
\end{eqnarray*}%
where
\begin{equation*}
g(tu_{0}):=\frac{at^{2}}{4}\left\Vert u_{0}\right\Vert _{D^{1,2}}^{4}+\frac{1%
}{2}\left( b\left\Vert u_{0}\right\Vert _{D^{1,2}}^{2}+\int_{\mathbb{R}%
^{N}}\lambda V(x)u_{0}^{2}dx\right) -\frac{t}{3}\int_{\mathbb{R}^{N}}q\left(
x\right) \left\vert u_{0}\right\vert ^{3}dx.
\end{equation*}%
A direct calculation shows that $\min_{t>0}g(tu_{0})<0$ by $\left( \ref{3-10}%
\right) .$ This indicates that there exists $t_{0}>0$ such that $J_{\lambda
,a}(t_{0}u_{0})<0.$ Letting $e=t_{0}u_{0}\in X_{\lambda }.$ Then $\Vert
e\Vert _{\lambda }>\rho ,$ since $\rho >0$ small enough. Hence, there exsits
$e\in X_{\lambda }$ with $\Vert e\Vert _{\lambda }>\rho $ such that $%
J_{\lambda ,a}(e)<0$ for all $0<a<\frac{2}{9\widehat{\mu }_{1}^{2}(\lambda )}
$ and $\lambda >\frac{S^{2}}{c_{0}\left\vert \left\{ V<c_{0}\right\}
\right\vert ^{1/2}}.$
\end{remark}

\begin{lemma}
\label{lem9}Suppose that conditions $(V1)-(V2),(L1),(D1)^{\prime }$ and $(D2)
$ hold. Let $\rho >0$ be as Lemma \ref{lem2}. Then for each $0<a<\frac{1}{%
m_{\infty }\overline{\mu }_{1}^{(k)}},$ there exists $e\in X$ with $\Vert
e\Vert _{\lambda }>\rho $ such that $J_{\lambda ,a}(e)<0$ for all $\lambda
>0.$
\end{lemma}

\begin{proof}
It follows from (\ref{1-2}) that for each $0<a<\frac{1}{m_{\infty }\overline{%
\mu }_{1}^{(k)}},$ there exists $\psi _{k}\in H^{1}(\mathbb{R}^{N})$ with $%
\psi _{k}>0$ such that
\begin{equation*}
\overline{\mu }_{1}^{(k)}\leq \frac{\left\Vert \psi _{k}\right\Vert
_{D^{1,2}}^{k}}{\int_{\mathbb{R}^{N}}q(x)\psi _{k}^{k}dx}<\frac{1}{%
am_{\infty }},
\end{equation*}%
which implies that%
\begin{equation*}
am_{\infty }\left\Vert \psi _{k}\right\Vert _{D^{1,2}}^{k}-\int_{\mathbb{R}%
^{N}}q(x)\psi _{k}^{k}dx<\frac{1}{\overline{\mu }_{1}^{(k)}}\left\Vert \psi
_{k}\right\Vert _{D^{1,2}}^{k}-\int_{\mathbb{R}^{N}}q(x)\psi _{k}^{k}dx\leq
0.
\end{equation*}%
Using this, together with conditions $(L1),(D1)^{\prime },(D2)$ and
Lebesgue's dominated convergence theorem, yields
\begin{eqnarray*}
\lim_{t\rightarrow +\infty }\frac{J_{\lambda ,a}(t\psi _{k})}{t^{k}}
&=&\lim_{t\rightarrow +\infty }\frac{1}{2t^{k-2}}\left( b\left\Vert \psi
_{k}\right\Vert _{D^{1,2}}^{2}+\int_{\mathbb{R}^{N}}\lambda V(x)\psi
_{k}^{2}dx\right)  \\
&&+\lim_{t\rightarrow +\infty }\left[ \frac{a\hat{m}\left( t^{2}\left\Vert
\psi _{k}\right\Vert _{D^{1,2}}^{2}\right) }{2t^{k}\left\Vert \psi
_{k}\right\Vert _{D^{1,2}}^{k}}\left\Vert \psi _{k}\right\Vert
_{D^{1,2}}^{k}-\int_{\mathbb{R}^{N}}\frac{F(x,t\psi _{k})}{t^{k}\psi _{k}^{k}%
}\psi _{k}^{k}dx\right]  \\
&\leq &\frac{am_{\infty }}{k}\left\Vert \psi _{k}\right\Vert _{D^{1,2}}^{k}-%
\frac{1}{k}\int_{\mathbb{R}^{N}}q(x)\psi _{k}^{k}dx \\
&=&\frac{1}{k}\left( am_{\infty }\left\Vert \psi _{k}\right\Vert
_{D^{1,2}}^{k}-\int_{\mathbb{R}^{N}}q(x)\psi _{k}^{k}dx\right) <0.
\end{eqnarray*}%
This implies that $J_{\lambda ,a}(t\psi _{k})\rightarrow -\infty $ as $%
t\rightarrow +\infty .$ Hence, for each $0<a<\frac{1}{m_{\infty }\overline{%
\mu }_{1}^{(k)}},$ there exists $e\in X$ with $\Vert e\Vert _{\lambda }>\rho
$ such that $J_{\lambda ,a}(e)<0$ for all $\lambda >0$ and the lemma is
proved.
\end{proof}

If condition $(L1)$ is not required, then we also have a conclusion similar
to Lemma \ref{lem9}, but $a>0$ must be small.

\begin{lemma}
\label{lem11}Suppose that conditions $(V1)-(V2),(D1)^{\prime }$ and $(D2)$
hold. Let $\rho >0$ be as Lemma \ref{lem2}. Then there exists $\overline{a}%
_{\ast }>0$ and $e\in X$ with $\Vert e\Vert _{\lambda }>\rho $ such that $%
J_{\lambda ,a}(e)<0$ for all $0<a<\overline{a}_{\ast }$ and $\lambda >0.$
\end{lemma}

\begin{proof}
The proof is similar to that of Lemma \ref{lem10}, and we omit it here.
\end{proof}

\begin{remark}
\label{r3} Similar to Remark \ref{re1}, the value of $\overline{a}_{\ast }$
can also not be determined in general, but only in some special cases. Next,
we give an example. For any real number $2<k<2^{\ast },$ we assume that $%
m\left( t\right) =m_{0}t^{\delta }$ with $\delta >\frac{k-2}{2}$ and%
\begin{equation*}
f\left( x,s\right) =\left\{
\begin{array}{ll}
q\left( x\right) s^{k-1} & \text{ if }s\geq 0, \\
0 & \text{ if }s<0,%
\end{array}%
\right.
\end{equation*}%
where the function $q$ satisfies condition $\left( D1\right) ^{\prime }.$ It
is easily seen that $m\left( t\right) $ does not satisfy condition $(L1).$
Let us consider the minimum problem:%
\begin{equation*}
\widetilde{\mu }_{1}^{(k)}(\lambda ):=\inf_{u\in H^{1}(\mathbb{R}^{N})}\frac{%
\left( \int_{\mathbb{R}^{N}}|\nabla u|^{2}dx\right) ^{\frac{(\delta +1)(k-2)%
}{2\delta -k+2}}\left( b\int_{\mathbb{R}^{N}}|\nabla u|^{2}dx+\int_{\mathbb{R%
}^{N}}\lambda V(x)u^{2}dx\right) }{\left( \int_{\mathbb{R}^{N}}q\left(
x\right) \left\vert u\right\vert ^{k}dx\right) ^{\frac{2\delta }{2\delta -k+2%
}}}.  \label{1-3}
\end{equation*}%
Then $\widetilde{\mu }_{1}^{(k)}(\lambda )\geq b\left( \overline{\nu }%
_{1}^{(k)}\right) ^{-\frac{2\delta }{2\delta -k+2}}>0$ for all $\lambda >0,$
where $\overline{\nu }_{1}^{(k)}$ is as (\ref{1-1}). Indeed, by condition $%
(D1)^{\prime }$ and the Caffarelli-Kohn-Nirenberg inequality one has%
\begin{equation*}
\int_{\mathbb{R}^{N}}q\left( x\right) \left\vert u\right\vert ^{k}dx\leq
\frac{c_{\ast }}{\overline{\nu }_{1}^{(k)}}\left( \int_{\mathbb{R}%
^{N}}|\nabla u|^{2}dx\right) ^{k/2}.
\end{equation*}%
Using the above inequality leads to for all $\lambda >0,$%
\begin{eqnarray*}
&&\frac{\left( \int_{\mathbb{R}^{N}}|\nabla u|^{2}dx\right) ^{\frac{(\delta
+1)(k-2)}{2\delta -k+2}}\left( b\int_{\mathbb{R}^{N}}|\nabla u|^{2}dx+\int_{%
\mathbb{R}^{N}}\lambda V(x)u^{2}dx\right) }{\left( \int_{\mathbb{R}%
^{N}}q\left( x\right) \left\vert u\right\vert ^{k}dx\right) ^{\frac{2\delta
}{2\delta -k+2}}} \\
&\geq &\frac{b\left( \int_{\mathbb{R}^{N}}|\nabla u|^{2}dx\right) ^{1+\frac{%
(\delta +1)(k-2)}{2\delta -k+2}}}{\left( \frac{c_{\ast }}{\overline{\nu }%
_{1}^{(k)}}\right) ^{\frac{2\delta }{2\delta -k+2}}\left( \int_{\mathbb{R}%
^{N}}|\nabla u|^{2}dx\right) ^{\frac{k\delta }{2\delta -k+2}}}=b\left( \frac{%
\overline{\nu }_{1}^{(k)}}{c_{\ast }}\right) ^{\frac{2\delta }{2\delta -k+2}%
}>0,
\end{eqnarray*}%
which implies that $\widetilde{\mu }_{1}^{(k)}(\lambda )\geq b\left( \frac{%
\overline{\nu }_{1}^{(k)}}{c_{\ast }}\right) ^{\frac{2\delta }{2\delta -k+2}%
}>0.$ Thus, for every $0<a<\frac{(\delta +1)(k-2)}{\delta m_{0}k}\left(
\frac{2\delta -k+2}{k\delta \widetilde{\mu }_{1}^{(k)}(\lambda )}\right) ^{%
\frac{2\delta -k+2}{k-2}},$ there exists $u_{0}\in X_{\lambda }$ such that%
\begin{eqnarray*}
&&\left( \frac{2\delta -k+2}{k\delta }\right) \left[ \frac{(\delta +1)(k-2)}{%
\delta m_{0}ak}\right] ^{\frac{k-2}{2\delta -k+2}}\left( \int_{\mathbb{R}%
^{N}}q\left( x\right) \left\vert u_{0}\right\vert ^{k}dx\right) ^{\frac{%
2\delta }{2\delta -k+2}} \\
&>&\left( \int_{\mathbb{R}^{N}}|\nabla u_{0}|^{2}dx\right) ^{\frac{(\delta
+1)(k-2)}{2\delta -k+2}}\left( b\int_{\mathbb{R}^{N}}|\nabla
u_{0}|^{2}dx+\int_{\mathbb{R}^{N}}\lambda V(x)u_{0}^{2}dx\right) .
\end{eqnarray*}%
Next, following the argument in Remark \ref{re1} we obtain that there exsits
$e\in X_{\lambda }$ with $\Vert e\Vert _{\lambda }>\rho $ such that $%
J_{\lambda ,a}(e)<0$ for all $0<a<\frac{(\delta +1)(k-2)}{\delta m_{0}k}%
\left( \frac{2\delta -k+2}{k\delta \widetilde{\mu }_{1}^{(k)}(\lambda )}%
\right) ^{\frac{2\delta -k+2}{k-2}}$ and $\lambda >0.$
\end{remark}

\section{Proofs of Theorems \protect\ref{t1} and \protect\ref{t2}}

Define%
\begin{equation*}
\alpha _{\lambda ,a}=\inf_{\gamma \in \Gamma _{\lambda }}\max_{0\leq t\leq
1}J_{\lambda ,a}(\gamma (t))
\end{equation*}%
and%
\begin{equation*}
\alpha _{0,a}\left( \Omega \right) =\inf_{\gamma \in \overline{\Gamma }%
_{\lambda }\left( \Omega \right) }\max_{0\leq t\leq 1}J_{\lambda
,a}|_{H_{0}^{1}\left( \Omega \right) }(\gamma (t)),
\end{equation*}%
where%
\begin{equation*}
\Gamma _{\lambda }=\{\gamma \in C([0,1],X_{\lambda }):\gamma (0)=0,\gamma
(1)=e\}
\end{equation*}%
and%
\begin{equation*}
\overline{\Gamma }_{\lambda }\left( \Omega \right) =\{\gamma \in
C([0,1],H_{0}^{1}\left( \Omega \right) ):\gamma (0)=0,\gamma (1)=e\}.
\end{equation*}%
Note that $\alpha _{0,a}\left( \Omega \right) $ independent of $\lambda .$
Following the argument in \cite{SW1}, we can take a number $D_{a}>0$ such
that $0<\eta \leq \alpha _{\lambda ,a}<\alpha _{0,a}(\Omega )<D_{a}$ for all
$\lambda \geq \Lambda _{N}$. Thus, by Lemmas \ref{lem1} and \ref{lem3} (or
Lemma \ref{lem10}) and the mountain pass theorem \cite{E2}, we obtain that
for each $\lambda \geq \Lambda _{N}$ and $a>0$ (or $0<a<\widetilde{{a}}%
_{\ast }$), there exists a sequence $\left\{ u_{n}\right\} \subset
X_{\lambda }$ such that
\begin{equation}
J_{\lambda ,a}(u_{n})\rightarrow \alpha _{\lambda ,a}>0\quad \text{and}\quad
(1+\Vert u_{n}\Vert _{\lambda })\Vert J_{\lambda ,a}^{\prime }(u_{n})\Vert
_{X_{\lambda }^{-1}}\rightarrow 0,\quad \text{as}\ n\rightarrow \infty ,
\label{3.5}
\end{equation}%
where $0<\eta \leq \alpha _{\lambda ,a}\leq \alpha _{0,a}\left( \Omega
\right) <D_{a}.$ Furthermore, we have the following results.

\begin{lemma}
\label{lem5}Suppose that conditions $(V1)-(V2),(L1)$ and $(D2)$ hold. Then
for each $a>0$ and $\lambda \geq \Lambda _{N},$ the sequence $\{u_{n}\}$
defined in (\ref{3.5}) is bounded in $X_{\lambda }.$
\end{lemma}

\begin{proof}
By condition $(D2),$ for $s>0$ one has%
\begin{eqnarray}
F(x,s)-\frac{1}{k}f(x,s)s &=&\int_{0}^{s}\frac{f(x,t)}{t^{k-1}}%
t^{k-1}dt-\int_{0}^{s}\frac{f(x,s)}{s^{k-1}}t^{k-1}dt  \notag \\
&=&\int_{0}^{s}\left( \frac{f(x,t)}{t^{k-1}}-\frac{f(x,s)}{s^{k-1}}\right)
t^{k-1}dt  \notag \\
&\leq &0.  \label{3.11}
\end{eqnarray}%
For $n$ large enough, it follows from condition $(L1)$ and $(\ref{3.5})-(\ref%
{3.11})$ that%
\begin{eqnarray*}
\alpha _{\lambda ,a}+1 &\geq &J_{\lambda ,a}\left( u_{n}\right) -\frac{1}{k}%
\langle J_{\lambda ,a}^{\prime }(u_{n}),u_{n}\rangle \\
&=&\frac{k-2}{2k}\left( b\int_{\mathbb{R}^{N}}\left\vert \nabla
u_{n}\right\vert ^{2}+\lambda V(x)u_{n}^{2}dx\right) \\
&&+\frac{a}{2}\left[ \widehat{m}\left( \left\Vert u_{n}\right\Vert
_{D^{1,2}}^{2}\right) -\frac{2}{k}m\left( \left\Vert u_{n}\right\Vert
_{D^{1,2}}^{2}\right) \left\Vert u_{n}\right\Vert _{D^{1,2}}^{2}\right] \\
&&-\int_{\mathbb{R}^{N}}\left[ F(x,u_{n})-\frac{1}{k}f(x,u_{n})u_{n}\right]
dx \\
&\geq &\frac{(k-2)\min \{b,1\}}{2k}\Vert u_{n}\Vert _{\lambda }^{2},
\end{eqnarray*}%
which implies that $\{u_{n}\}$ is bounded in $X_{\lambda }$ for each $a>0$
and $\lambda \geq \Lambda _{N}.$
\end{proof}

\begin{lemma}
\label{lem7}Suppose that $N\geq 3$, conditions $(V1)-(V2),(L3)$ with $\delta
\geq \frac{2}{N-2}$ and $(D1)-(D2)$ hold. Then for all $0<a<\widetilde{{a}}%
_{\ast }$ and
\begin{equation*}
\lambda >\widetilde{{\Lambda }}_{0}:=\left\{
\begin{array}{ll}
\frac{|q|_{\infty }}{c_{0}}\max \left\{ \left( \frac{|q|_{\infty }}{am_{0}%
\overline{S}^{2^{\ast }}}\right) ^{\frac{k-2}{2^{\ast }-k}},\frac{2^{\ast }-k%
}{2^{\ast }-2}\right\} & \text{ if }\delta =\frac{2}{N-2}, \\
\frac{|q|_{\infty }\left( 2^{\ast }-k\right) }{c_{0}(2^{\ast }-2)} & \text{
if }\delta >\frac{2}{N-2},%
\end{array}%
\right.
\end{equation*}%
the sequence $\{u_{n}\}$ defined in (\ref{3.5}) is bounded in $X_{\lambda }.$
\end{lemma}

\begin{proof}
$\left( i\right) $ $\delta =\frac{2}{N-2}:$ Note that $2(\delta +1)=2^{\ast
}.$ Suppose on the contrary. Then $\Vert u_{n}\Vert _{\lambda }\rightarrow
+\infty $ as $n\rightarrow \infty $. The proof is divided into three
separate cases:

Case $A:\Vert u_{n}\Vert _{D^{1,2}}\rightarrow \infty $ and
\begin{equation}
\frac{\int_{\mathbb{R}^{N}}\lambda V(x)u_{n}^{2}dx}{\left\Vert
u_{n}\right\Vert _{D^{1,2}}^{2(\delta +1)}}\geq \lambda c_{0}S^{-2^{\ast
}}\left( \frac{|q|_{\infty }}{\lambda c_{0}}\right) ^{\frac{2^{\ast }-2}{k-2}%
}.  \label{14-3}
\end{equation}%
By (\ref{3.5}), we have%
\begin{equation*}
\frac{\langle J_{a,\lambda }^{\prime }(u_{n}),u_{n}\rangle }{\Vert
u_{n}\Vert _{D^{1,2}}^{2(\delta +1)}}=o(1),
\end{equation*}%
where $o(1)$ denotes a quantity which goes to zero as $n\rightarrow \infty .$
Using this, together with condition $(L3)$ and (\ref{3.1}), gives%
\begin{eqnarray}
o(1) &=&\frac{am\left( \left\Vert u_{n}\right\Vert _{D^{1,2}}^{2}\right)
\left\Vert u_{n}\right\Vert _{D^{1,2}}^{2}}{\Vert u_{n}\Vert
_{D^{1,2}}^{2(\delta +1)}}+\frac{b\left\Vert u_{n}\right\Vert
_{D^{1,2}}^{2}+\int_{\mathbb{R}^{N}}\lambda V(x)u_{n}^{2}dx}{\Vert
u_{n}\Vert _{D^{1,2}}^{2(\delta +1)}}-\frac{\int_{\mathbb{R}%
^{N}}f(x,u_{n})u_{n}dx}{\Vert u_{n}\Vert _{D^{1,2}}^{2(\delta +1)}}  \notag
\\
&\geq &am_{0}+\frac{b}{\left\Vert u_{n}\right\Vert _{D^{1,2}}^{2\delta }}+%
\frac{\int_{\mathbb{R}^{N}}\lambda V(x)u_{n}^{2}dx-|q|_{\infty }\int_{%
\mathbb{R}^{N}}|u_{n}|^{k}dx}{\Vert u_{n}\Vert _{D^{1,2}}^{2(\delta +1)}}.
\label{14-6}
\end{eqnarray}%
Note that%
\begin{eqnarray*}
&&\int_{\mathbb{R}^{N}}|u_{n}|^{k}dx \\
&\leq &\left( \frac{1}{\lambda c_{0}}\int_{\mathbb{R}^{N}}\lambda
V(x)u_{n}^{2}dx+\left\vert \left\{ V<c_{0}\right\} \right\vert
^{2/N}S^{-2}\left\Vert u_{n}\right\Vert _{D^{1,2}}^{2}\right) ^{\frac{%
2^{\ast }-k}{2^{\ast }-2}}\left( S^{-2^{\ast }}\left\Vert u_{n}\right\Vert
_{D^{1,2}}^{2^{\ast }}\right) ^{\frac{k-2}{2^{\ast }-2}} \\
&\leq &S^{-\frac{2^{\ast }(k-2)}{2^{\ast }-2}}(\lambda c_{0})^{-\frac{%
2^{\ast }-k}{2^{\ast }-2}}\left( \int_{\mathbb{R}^{N}}\lambda
V(x)u_{n}^{2}dx\right) ^{\frac{2^{\ast }-k}{2^{\ast }-2}}\left\Vert
u_{n}\right\Vert _{D^{1,2}}^{\frac{2^{\ast }(k-2)}{2^{\ast }-2}} \\
&&+\left\vert \left\{ V<c_{0}\right\} \right\vert ^{\left( 2^{\ast
}-k\right) /2^{\ast }}S^{-k}\left\Vert u_{n}\right\Vert _{D^{1,2}}^{k}.
\end{eqnarray*}%
Then there holds%
\begin{eqnarray}
&&\frac{\int_{\mathbb{R}^{N}}\lambda V(x)u_{n}^{2}dx-|q|_{\infty }\int_{%
\mathbb{R}^{N}}|u_{n}|^{k}dx}{\Vert u_{n}\Vert _{D^{1,2}}^{2(\delta +1)}}
\notag \\
&\geq &\frac{\int_{\mathbb{R}^{N}}\lambda V(x)u_{n}^{2}dx}{\Vert u_{n}\Vert
_{D^{1,2}}^{2(\delta +1)}}\cdot \left[ 1-|q|_{\infty }(\lambda c_{0})^{-%
\frac{2^{\ast }-k}{2^{\ast }-2}}\left( \frac{S^{-2^{\ast }}\left\Vert
u_{n}\right\Vert _{D^{1,2}}^{2^{\ast }}}{\int_{\mathbb{R}^{N}}\lambda
V(x)u_{n}^{2}dx}\right) ^{\frac{k-2}{2^{\ast }-2}}\right]   \notag \\
&&-\frac{|q|_{\infty }\left\vert \left\{ V<c_{0}\right\} \right\vert
^{\left( 2^{\ast }-k\right) /2^{\ast }}}{S^{k}\left\Vert u_{n}\right\Vert
_{D^{1,2}}^{2(\delta +1)-k}}.  \label{14-1}
\end{eqnarray}%
It follows from (\ref{14-3})-(\ref{14-1}) that%
\begin{eqnarray*}
o\left( 1\right)  &\geq &am_{0}+\frac{b}{\left\Vert u_{n}\right\Vert
_{D^{1,2}}^{2\delta }}-\frac{|q|_{\infty }\left\vert \left\{ V<c_{0}\right\}
\right\vert ^{\left( 2^{\ast }-k\right) /2^{\ast }}}{S^{k}\left\Vert
u_{n}\right\Vert _{D^{1,2}}^{2(\delta +1)-k}} \\
&&+\frac{\int_{\mathbb{R}^{N}}\lambda V(x)u_{n}^{2}dx}{\Vert u_{n}\Vert
_{D^{1,2}}^{2(\delta +1)}}\cdot \left[ 1-|q|_{\infty }(\lambda c_{0})^{-%
\frac{2^{\ast }-k}{2^{\ast }-2}}\left( \frac{S^{-2^{\ast }}\left\Vert
u_{n}\right\Vert _{D^{1,2}}^{2^{\ast }}}{\int_{\mathbb{R}^{N}}\lambda
V(x)u_{n}^{2}dx}\right) ^{\frac{k-2}{2^{\ast }-2}}\right]  \\
&=&am_{0}+\frac{\int_{\mathbb{R}^{N}}\lambda V(x)u_{n}^{2}dx}{\Vert
u_{n}\Vert _{D^{1,2}}^{2(\delta +1)}}\left[ 1-|q|_{\infty }(\lambda c_{0})^{-%
\frac{2^{\ast }-k}{2^{\ast }-2}}\left( \frac{S^{-2^{\ast }}\left\Vert
u_{n}\right\Vert _{D^{1,2}}^{2^{\ast }}}{\int_{\mathbb{R}^{N}}\lambda
V(x)u_{n}^{2}dx}\right) ^{\frac{k-2}{2^{\ast }-2}}\right] +o\left( 1\right)
\\
&\geq &am_{0}+o\left( 1\right) .
\end{eqnarray*}%
This is a contradiction.

Case $B:\Vert u_{n}\Vert _{D^{1,2}}\rightarrow \infty $ and
\begin{equation}
\frac{\int_{\mathbb{R}^{N}}\lambda V(x)u_{n}^{2}dx}{\left\Vert
u_{n}\right\Vert _{D^{1,2}}^{2(\delta +1)}}<\lambda c_{0}S^{-2^{\ast
}}\left( \frac{|q|_{\infty }}{\lambda c_{0}}\right) ^{\frac{2^{\ast }-2}{k-2}%
}.  \label{14-4}
\end{equation}%
Applying (\ref{10}) and (\ref{14-4})$\ $leads to%
\begin{eqnarray}
&&\frac{\int_{\mathbb{R}^{N}}|u_{n}|^{k}dx}{\left\Vert u_{n}\right\Vert
_{D^{1,2}}^{2(\delta +1)}}  \notag \\
&\leq &\frac{\left( \frac{1}{\lambda c_{0}}\int_{\mathbb{R}^{N}}\lambda
V(x)u_{n}^{2}dx+\left\vert \left\{ V<c_{0}\right\} \right\vert
^{2/N}S^{-2}\left\Vert u_{n}\right\Vert _{D^{1,2}}^{2}\right) ^{\frac{%
2^{\ast }-k}{2^{\ast }-2}}\left( S^{-2^{\ast }}\left\Vert u_{n}\right\Vert
_{D^{1,2}}^{2^{\ast }}\right) ^{\frac{k-2}{2^{\ast }-2}}}{\left\Vert
u_{n}\right\Vert _{D^{1,2}}^{2(\delta +1)}}  \notag \\
&\leq &\frac{\left[ \left( \frac{1}{\lambda c_{0}}\int_{\mathbb{R}%
^{N}}\lambda V(x)u_{n}^{2}dx\right) ^{\frac{2^{\ast }-k}{2^{\ast }-2}%
}+\left( \left\vert \left\{ V<c_{0}\right\} \right\vert
^{2/N}S^{-2}\left\Vert u_{n}\right\Vert _{D^{1,2}}^{2}\right) ^{\frac{%
2^{\ast }-k}{2^{\ast }-2}}\right] \left( S^{-2^{\ast }}\left\Vert
u_{n}\right\Vert _{D^{1,2}}^{2^{\ast }}\right) ^{\frac{k-2}{2^{\ast }-2}}}{%
\left\Vert u_{n}\right\Vert _{D^{1,2}}^{2(\delta +1)}}  \notag \\
&\leq &S^{-2^{\ast }}\left( \frac{|q|_{\infty }}{\lambda c_{0}}\right) ^{%
\frac{2^{\ast }-k}{k-2}}+\frac{S^{-k}\left\vert \left\{ V<c_{0}\right\}
\right\vert ^{\frac{2^{\ast }-k}{2^{\ast }}}}{\left\Vert u_{n}\right\Vert
_{D^{1,2}}^{2(\delta +1)-k}}  \notag \\
&=&S^{-2^{\ast }}\left( \frac{|q|_{\infty }}{\lambda c_{0}}\right) ^{\frac{%
2^{\ast }-k}{k-2}}+o\left( 1\right) .  \label{14-5}
\end{eqnarray}%
By (\ref{14-6}) and (\ref{14-5}) one has%
\begin{eqnarray*}
o\left( 1\right) =\frac{\langle J_{\lambda ,a}^{\prime }(u_{n}),u_{n}\rangle
}{\left\Vert u_{n}\right\Vert _{D^{1,2}}^{2(\delta +1)}} &\geq &am_{0}+\frac{%
b}{\left\Vert u_{n}\right\Vert _{D^{1,2}}^{2\delta }}+\frac{\int_{\mathbb{R}%
^{N}}\lambda V(x)u_{n}^{2}dx}{\Vert u_{n}\Vert _{D^{1,2}}^{2(\delta +1)}}-%
\frac{|q|_{\infty }\int_{\mathbb{R}^{N}}|u_{n}|^{k}dx}{\Vert u_{n}\Vert
_{D^{1,2}}^{2(\delta +1)}} \\
&\geq &am_{0}+\frac{\int_{\mathbb{R}^{N}}\lambda V(x)u_{n}^{2}dx}{\Vert
u_{n}\Vert _{D^{1,2}}^{2(\delta +1)}}-|q|_{\infty }S^{-2^{\ast }}\left(
\frac{|q^{+}|_{\infty }}{\lambda c_{0}}\right) ^{\frac{2^{\ast }-k}{k-2}%
}+o\left( 1\right)  \\
&\geq &am_{0}-|q|_{\infty }S^{-2^{\ast }}\left( \frac{|q|_{\infty }}{\lambda
c_{0}}\right) ^{\frac{2^{\ast }-k}{k-2}}+o\left( 1\right)
\end{eqnarray*}%
which contradicts with%
\begin{equation*}
\lambda >\frac{|q|_{\infty }}{c_{0}}\left( \frac{|q|_{\infty }}{%
am_{0}S^{2^{\ast }}}\right) ^{\frac{k-2}{2^{\ast }-k}}.
\end{equation*}

Case $C:\int_{\mathbb{R}^{N}}\lambda V(x)u_{n}^{2}dx\rightarrow \infty $ and
$\left\Vert u_{n}\right\Vert _{D^{1,2}}\leq C_{\ast }$ for some $C_{\ast }>0$
and for all $n.$ From (\ref{10}), (\ref{13}) and condition $(L3)$ it follows
that%
\begin{eqnarray}
o(1) &=&\frac{am\left( \left\Vert u\right\Vert _{D^{1,2}}^{2}\right)
\left\Vert u\right\Vert _{D^{1,2}}^{2}}{\int_{\mathbb{R}^{N}}\lambda
V(x)u_{n}^{2}dx}+\frac{b\left\Vert u\right\Vert _{D^{1,2}}^{2}+\int_{\mathbb{%
R}^{N}}\lambda V(x)u^{2}dx}{\int_{\mathbb{R}^{N}}\lambda V(x)u_{n}^{2}dx}-%
\frac{\int_{\mathbb{R}^{N}}f(x,u_{n})u_{n}dx}{\int_{\mathbb{R}^{N}}\lambda
V(x)u_{n}^{2}dx}  \notag \\
&\geq &\frac{b\left\Vert u_{n}\right\Vert _{D^{1,2}}^{2}}{\int_{\mathbb{R}%
^{N}}\lambda V(x)u_{n}^{2}dx}+1-\frac{|q|_{\infty }\int_{\mathbb{R}%
^{N}}\left\vert u_{n}\right\vert ^{k}dx}{\int_{\mathbb{R}^{N}}\lambda
V(x)u_{n}^{2}dx}.  \label{14-2}
\end{eqnarray}%
Applying (\ref{10}) and the Young inequality gives%
\begin{eqnarray}
\int_{\mathbb{R}^{N}}|u_{n}|^{k}dx &\leq &\frac{2^{\ast }-k}{2^{\ast }-2}%
\left( \frac{1}{\lambda c_{0}}\int_{\mathbb{R}^{N}}\lambda
V(x)u_{n}^{2}dx+\left\vert \left\{ V<c_{0}\right\} \right\vert
^{2/N}S^{-2}\Vert u_{n}\Vert _{D^{1,2}}^{2}\right)  \notag \\
&&+\frac{k-2}{2^{\ast }-2}S^{-2^{\ast }}\Vert u_{n}\Vert _{D^{1,2}}^{2^{\ast
}}  \label{14-7}
\end{eqnarray}%
By (\ref{14-7}) and the fact of $\Vert u_{n}\Vert _{D^{1,2}}\leq C_{\ast }$
for all $n,$ we obtain that
\begin{eqnarray}
\frac{\int_{\mathbb{R}^{N}}|u_{n}|^{k}dx}{\int_{\mathbb{R}^{N}}\lambda
V(x)u_{n}^{2}dx} &\leq &\frac{\left( 2^{\ast }-k\right) }{(2^{\ast
}-2)\lambda c_{0}}+\frac{\left( 2^{\ast }-k\right) \left\vert \left\{
V<c_{0}\right\} \right\vert ^{2/N}S^{-2}C_{\ast }^{2}+(k-2)S^{-2^{\ast
}}C_{\ast }^{2^{\ast }}}{(2^{\ast }-2)\int_{\mathbb{R}^{N}}\lambda
V(x)u_{n}^{2}dx}  \notag \\
&=&\frac{2^{\ast }-k}{(2^{\ast }-2)\lambda c_{0}}+o\left( 1\right) .
\label{14-9}
\end{eqnarray}%
Using (\ref{14-2}) and (\ref{14-9}) yields%
\begin{eqnarray*}
o(1) &\geq &\frac{b\left\Vert u_{n}\right\Vert _{D^{1,2}}^{2}}{\int_{\mathbb{%
R}^{N}}\lambda V(x)u_{n}^{2}dx}+1-\frac{|q|_{\infty }\int_{\mathbb{R}%
^{N}}\left\vert u_{n}\right\vert ^{p}dx}{\int_{\mathbb{R}^{N}}\lambda
V(x)u_{n}^{2}dx} \\
&\geq &1-\frac{|q|_{\infty }\left( 2^{\ast }-k\right) }{(2^{\ast }-2)\lambda
c_{0}}+o\left( 1\right) ,
\end{eqnarray*}%
which contradicts with%
\begin{equation*}
\lambda >\frac{\left( 2^{\ast }-k\right) |q|_{\infty }}{(2^{\ast }-2)c_{0}}.
\end{equation*}%
$\left( ii\right) $ $\delta >\frac{2}{N-2}:$ Clearly, $2(\delta +1)>2^{\ast
}.$ Suppose on the contrary. Then $\Vert u_{n}\Vert _{\lambda }\rightarrow
+\infty $ as $n\rightarrow \infty $. We consider the proof in two separate
cases:

Case $D:\Vert u_{n}\Vert _{D^{1,2}}\rightarrow \infty .$ It follows from (%
\ref{13})--(\ref{3.1}) and condition $(L3)$ that%
\begin{eqnarray}
o(1) &=&\frac{am\left( \left\Vert u_{n}\right\Vert _{D^{1,2}}^{2}\right)
\left\Vert u_{n}\right\Vert _{D^{1,2}}^{2}}{\Vert u_{n}\Vert
_{D^{1,2}}^{2^{\ast }}}+\frac{b\left\Vert u_{n}\right\Vert
_{D^{1,2}}^{2}+\int_{\mathbb{R}^{N}}\lambda V(x)u_{n}^{2}dx}{\Vert
u_{n}\Vert _{D^{1,2}}^{2^{\ast }}}-\frac{\int_{\mathbb{R}%
^{N}}f(x,u_{n})u_{n}dx}{\Vert u_{n}\Vert _{D^{1,2}}^{2^{\ast }}}  \notag \\
&\geq &am_{0}\Vert u_{n}\Vert _{D^{1,2}}^{2(\delta +1)-2^{\ast }}+\frac{b}{%
\left\Vert u_{n}\right\Vert _{D^{1,2}}^{2^{\ast }-2}}+\frac{\int_{\mathbb{R}%
^{N}}\lambda V(x)u_{n}^{2}dx-|q|_{\infty }\int_{\mathbb{R}^{N}}\left\vert
u_{n}\right\vert ^{k}dx}{\left\Vert u_{n}\right\Vert _{D^{1,2}}^{2^{\ast }}}.
\label{3-6}
\end{eqnarray}%
By (\ref{14-7}), we deduce that%
\begin{eqnarray*}
&&\frac{\int_{\mathbb{R}^{N}}\lambda V(x)u_{n}^{2}dx-|q|_{\infty }\int_{%
\mathbb{R}^{N}}\left\vert u_{n}\right\vert ^{k}dx}{\left\Vert
u_{n}\right\Vert _{D^{1,2}}^{2^{\ast }}} \\
&\geq &\left( 1-\frac{|q|_{\infty }\left( 2^{\ast }-k\right) }{\lambda
c_{0}\left( 2^{\ast }-2\right) }\right) \frac{\int_{\mathbb{R}^{N}}\lambda
V(x)u_{n}^{2}dx}{\left\Vert u_{n}\right\Vert _{D^{1,2}}^{2^{\ast }}} \\
&&-\frac{|q|_{\infty }S^{-2}\left( 2^{\ast }-k\right) \left\vert \left\{
V<c_{0}\right\} \right\vert ^{2/N}}{(2^{\ast }-2)\left\Vert u_{n}\right\Vert
_{D^{1,2}}^{2^{\ast }-2}}-|q|_{\infty }S^{-2^{\ast }}\left( \frac{k-2}{%
2^{\ast }-2}\right) \\
&\geq &-\frac{|q|_{\infty }S^{-2}\left( 2^{\ast }-k\right) \left\vert
\left\{ V<c_{0}\right\} \right\vert ^{2/N}}{(2^{\ast }-2)\left\Vert
u_{n}\right\Vert _{D^{1,2}}^{2^{\ast }-2}}-|q|_{\infty }S^{-2^{\ast }}\left(
\frac{k-2}{2^{\ast }-2}\right) .
\end{eqnarray*}%
Using this, together with (\ref{3-6}), leads to
\begin{eqnarray*}
o(1) &\geq &am_{0}\Vert u_{n}\Vert _{D^{1,2}}^{2(\delta +1)-2^{\ast }}+\frac{%
b}{\left\Vert u_{n}\right\Vert _{D^{1,2}}^{2^{\ast }-2}}+\frac{\int_{\mathbb{%
R}^{N}}\lambda V(x)u_{n}^{2}dx-|q|_{\infty }\int_{\mathbb{R}^{N}}\left\vert
u_{n}\right\vert ^{k}dx}{\left\Vert u_{n}\right\Vert _{D^{1,2}}^{2^{\ast }}}
\\
&\geq &am_{0}\Vert u_{n}\Vert _{D^{1,2}}^{2(\delta +1)-2^{\ast }}-\frac{%
|q|_{\infty }S^{-2}\left( 2^{\ast }-k\right) \left\vert \left\{
V<c_{0}\right\} \right\vert ^{2/N}}{(2^{\ast }-2)\left\Vert u_{n}\right\Vert
_{D^{1,2}}^{2^{\ast }-2}}-|q|_{\infty }S^{-2^{\ast }}\left( \frac{k-2}{%
2^{\ast }-2}\right) \\
&\rightarrow &\infty \text{ as }n\rightarrow \infty ,
\end{eqnarray*}%
since $2(\delta +1)>2^{\ast }.$ This is a contradiction.

Case $E:\int_{\mathbb{R}^{N}}\lambda V(x)u_{n}^{2}dx\rightarrow \infty $ and
$\left\Vert u_{n}\right\Vert _{D^{1,2}}\leq C_{\ast }$ for some $C_{\ast }>0$
and for all $n.$ It follows from (\ref{13})--(\ref{3.1}) and condition $(L3)$
that%
\begin{eqnarray}
o(1) &=&\frac{am\left( \left\Vert u_{n}\right\Vert _{D^{1,2}}^{2}\right)
\left\Vert u_{n}\right\Vert _{D^{1,2}}^{2}}{\int_{\mathbb{R}^{N}}\lambda
V(x)u_{n}^{2}dx}+\frac{b\left\Vert u_{n}\right\Vert _{D^{1,2}}^{2}+\int_{%
\mathbb{R}^{N}}\lambda V(x)u_{n}^{2}dx}{\int_{\mathbb{R}^{N}}\lambda
V(x)u_{n}^{2}dx}-\frac{\int_{\mathbb{R}^{N}}f(x,u_{n})u_{n}dx}{\int_{\mathbb{%
R}^{N}}\lambda V(x)u_{n}^{2}dx}  \notag \\
&\geq &\frac{b\left\Vert u_{n}\right\Vert _{D^{1,2}}^{2}}{\int_{\mathbb{R}%
^{N}}\lambda V(x)u_{n}^{2}dx}+1-\frac{|q|_{\infty }\int_{\mathbb{R}%
^{N}}\left\vert u_{n}\right\vert ^{k}dx}{\int_{\mathbb{R}^{N}}\lambda
V(x)u_{n}^{2}dx}.  \label{3.7}
\end{eqnarray}%
By (\ref{14-9}) and (\ref{3.7}) one has%
\begin{eqnarray*}
o(1) &\geq &\frac{b\left\Vert u_{n}\right\Vert _{D^{1,2}}^{2}}{\int_{\mathbb{%
R}^{N}}\lambda V(x)u_{n}^{2}dx}+1-\frac{|q|_{\infty }\int_{\mathbb{R}%
^{N}}\left\vert u_{n}\right\vert ^{k}dx}{\int_{\mathbb{R}^{N}}\lambda
V(x)u_{n}^{2}dx} \\
&\geq &1-\frac{|q|_{\infty }\left( 2^{\ast }-k\right) }{\lambda c_{0}\left(
2^{\ast }-2\right) }+o\left( 1\right) ,
\end{eqnarray*}%
which contradicts with%
\begin{equation*}
\lambda >\frac{|q|_{\infty }\left( 2^{\ast }-k\right) }{c_{0}(2^{\ast }-2)}.
\end{equation*}%
In conclusion, the sequence $\{u_{n}\}$ is bounded in $X_{\lambda }$ for all
$0<a<\widetilde{{a}}_{\ast }$ and $\lambda >\widetilde{{\Lambda }}_{0}.$
This completes the proof.
\end{proof}

We now investigate the following two compactness results for the functional $%
J_{\lambda ,a}$ under conditions $(D1)-(D2)$.

\begin{proposition}
\label{m2} Suppose that $N\geq 1,$ conditions $(V1)-(V2),(L1)$ and $%
(D1)-(D2) $ hold. Then for each $D>0$ there exists $\widetilde{{\Lambda }}=%
\widetilde{{\Lambda }}(a,D)\geq \Lambda _{N}>0$ such that $J_{\lambda ,a}$
satisfies the $(C)_{\alpha }$--condition in $X_{\lambda }$ for all $\alpha
<D $ and $\lambda >\widetilde{{\Lambda }}.$
\end{proposition}

\begin{proof}
Let $\left\{ u_{n}\right\} $ be a $\left( C\right) _{\alpha }$-sequence with
$\alpha <D.$ By Lemma \ref{lem5}, we have $\left\{ u_{n}\right\} $ is
bounded in $X_{\lambda }.$ Then there exist a subsequence $\left\{
u_{n}\right\} $ and $u_{0}\in X_{\lambda }$ such that%
\begin{equation*}
u_{n}\rightharpoonup u_{0}\text{ weakly in }X_{\lambda }\quad \text{and}%
\quad u_{n}\rightarrow u_{0}\text{ strongly in }L_{loc}^{r}(\mathbb{R}^{N})%
\text{ for }2\leq r<2^{\ast }.
\end{equation*}%
Next, we prove that $u_{n}\rightarrow u_{0}$ strongly in $X_{\lambda }.$ Let
$v_{n}=u_{n}-u_{0}.$ Using condition $(V1)$ gives%
\begin{equation}
\int_{\mathbb{R}^{N}}v_{n}^{2}dx\leq \frac{1}{\lambda c_{0}}\left\Vert
v_{n}\right\Vert _{\lambda }^{2}+o(1).  \label{16}
\end{equation}%
Using this, together with the H\"{o}lder and Sobolev inequalities, for any $%
\lambda >\Lambda _{N}$, we check the following estimation:\newline
Case $(i)$ $N=1,2:$%
\begin{eqnarray*}
\int_{\mathbb{R}^{N}}\left\vert v_{n}\right\vert ^{r}dx &\leq &\left( \int_{%
\mathbb{R}^{N}}v_{n}^{2}dx\right) ^{1/2}\cdot \left( \int_{\mathbb{R}%
^{N}}v_{n}^{2(r-1)}dx\right) ^{1/2} \\
&\leq &\left[ \frac{1}{c_{0}\lambda }\left( 1+A_{N}^{8/N}\left\vert \left\{
V<c_{0}\right\} \right\vert ^{2/N}\right) ^{r-1}\right]
^{1/2}S_{2(r-1)}^{1-r}\left\Vert v_{n}\right\Vert _{\lambda }^{r}+o(1).
\end{eqnarray*}%
Case $(ii)$ $N\geq 3:$%
\begin{eqnarray*}
\int_{\mathbb{R}^{N}}\left\vert v_{n}\right\vert ^{r}dx &\leq &\left( \int_{%
\mathbb{R}^{N}}\left\vert v_{n}\right\vert ^{2}dx\right) ^{\frac{2^{\ast }-r%
}{2^{\ast }-2}}\left( \int_{\mathbb{R}^{N}}\left\vert v_{n}\right\vert
^{2^{\ast }}dx\right) ^{\frac{r-2}{2^{\ast }-2}} \\
&\leq &\left( \frac{1}{\lambda c_{0}}\right) ^{\frac{2^{\ast }-r}{2^{\ast }-2%
}}S^{-\frac{2^{\ast }\left( r-2\right) }{2^{\ast }-2}}\left\Vert
v_{n}\right\Vert _{\lambda }^{r}+o\left( 1\right) .
\end{eqnarray*}%
Set
\begin{equation*}
\Pi _{\lambda ,r}=\left\{
\begin{array}{ll}
\left[ \frac{1}{\lambda c_{0}}\left( 1+A_{N}^{8/N}\left\vert \left\{
V<c_{0}\right\} \right\vert ^{2/N}\right) ^{r-1}\right]
^{1/2}S_{2(r-1)}^{1-r} & \text{ if }N=1,2, \\
\left( \frac{1}{\lambda c_{0}}\right) ^{\frac{2^{\ast }-r}{2^{\ast }-2}}S^{-%
\frac{2^{\ast }\left( r-2\right) }{2^{\ast }-2}} & \text{ if }N\geq 3.%
\end{array}%
\right.
\end{equation*}%
Clearly, $\Pi _{\lambda ,r}\rightarrow 0$ as $\lambda \rightarrow \infty .$
Then we have%
\begin{equation}
\int_{\mathbb{R}^{N}}\left\vert v_{n}\right\vert ^{r}dx\leq \Pi _{\lambda
,r}\left\Vert v_{n}\right\Vert _{\lambda }^{r}+o\left( 1\right) .  \label{17}
\end{equation}%
Following the argument of \cite{SW2}, it is easy to verify that%
\begin{equation}
\int_{\mathbb{R}^{N}}F(x,v_{n})dx=\int_{\mathbb{R}^{N}}F(x,u_{n})dx-\int_{%
\mathbb{R}^{N}}F(x,u_{0})dx+o(1)  \label{14}
\end{equation}%
and%
\begin{equation*}
\sup_{\left\Vert h\right\Vert _{\lambda }=1}\int_{\mathbb{R}%
^{N}}[f(x,v_{n})-f(x,u_{n})+f(x,u_{0})]h(x)dx=o(1).
\end{equation*}%
Thus, using (\ref{12}), (\ref{14}) and Brezis-Lieb Lemma \cite{BL}, we
deduce that%
\begin{eqnarray*}
J_{a,\lambda }\left( u_{n}\right) -J_{a,\lambda }\left( u_{0}\right) &=&%
\frac{a}{2}\left[ \widehat{m}\left( \left\Vert u_{n}\right\Vert
_{D^{1,2}}^{2}\right) -\widehat{m}\left( \left\Vert u_{0}\right\Vert
_{D^{1,2}}^{2}\right) \right] +\frac{b}{2}\left\Vert v_{n}\right\Vert
_{D^{1,2}}^{2} \\
&&+\frac{1}{2}\int_{\mathbb{R}^{N}}\lambda V(x)v_{n}^{2}dx-\int_{\mathbb{R}%
^{N}}F(x,v_{n})dx+o\left( 1\right) .
\end{eqnarray*}%
Moreover, it follows from the boundedness of the sequence $\left\{
u_{n}\right\} $ in $X_{\lambda }$ that there exists a constant $A>0$ such
that $\left\Vert u_{n}\right\Vert _{D^{1,2}}^{2}\rightarrow A$ as $%
n\rightarrow \infty $, which indicates that for any $\varphi \in
C_{0}^{\infty }(\mathbb{R}^{N}),$ there holds
\begin{eqnarray*}
o(1) &=&\left\langle J_{a,\lambda }^{\prime }(u_{n}),\varphi \right\rangle \\
&=&\int_{\mathbb{R}^{N}}\lambda V(x)u_{n}\varphi dx+\left( am\left(
\left\Vert u_{n}\right\Vert _{D^{1,2}}^{2}\right) +b\right) \int_{\mathbb{R}%
^{N}}\nabla u_{n}\nabla \varphi dx-\int_{\mathbb{R}^{N}}f(x,u_{n})\varphi dx
\\
&\rightarrow &\int_{\mathbb{R}^{N}}\lambda V(x)u_{0}\varphi dx+(am\left(
A\right) +b)\int_{\mathbb{R}^{N}}\nabla u_{0}\nabla \varphi dx-\int_{\mathbb{%
R}^{N}}f(x,u_{0})\varphi dx
\end{eqnarray*}%
as $n\rightarrow \infty $. This implies that%
\begin{equation}
\int_{\mathbb{R}^{N}}\lambda V(x)u_{0}^{2}dx+\left( am\left( A\right)
+b\right) \left\Vert u_{0}\right\Vert _{D^{1,2}}^{2}-\int_{\mathbb{R}%
^{N}}f(x,u_{0})u_{0}dx=0.  \label{15}
\end{equation}%
Note that%
\begin{equation*}
o(1)=\int_{\mathbb{R}^{N}}\lambda V(x)u_{n}^{2}dx+(am\left( \left\Vert
u_{n}\right\Vert _{D^{1,2}}^{2}\right) +b)\left\Vert u_{n}\right\Vert
_{D^{1,2}}^{2}-\int_{\mathbb{R}^{N}}f(x,u_{n})u_{n}dx.
\end{equation*}%
Combining the above two equalities gives%
\begin{eqnarray}
o\left( 1\right) &=&\int_{\mathbb{R}^{N}}\lambda V(x)v_{n}^{2}dx+am\left(
\left\Vert u_{n}\right\Vert _{D^{1,2}}^{2}\right) \left\Vert
u_{n}\right\Vert _{D^{1,2}}^{2}-am\left( A\right) \left\Vert
u_{0}\right\Vert _{D^{1,2}}^{2}  \notag \\
&&+b\left\Vert v_{n}\right\Vert _{D^{1,2}}^{2}-\int_{\mathbb{R}%
^{N}}f(x,v_{n})v_{n}dx  \notag \\
&=&\int_{\mathbb{R}^{N}}\lambda V(x)v_{n}^{2}dx+am\left( \left\Vert
u_{n}\right\Vert _{D^{1,2}}^{2}\right) \left\Vert u_{n}\right\Vert
_{D^{1,2}}^{2}  \notag \\
&&-am\left( \left\Vert u_{n}\right\Vert _{D^{1,2}}^{2}\right) \left\Vert
u_{0}\right\Vert _{D^{1,2}}^{2}+b\left\Vert v_{n}\right\Vert
_{D^{1,2}}^{2}-\int_{\mathbb{R}^{N}}f(x,v_{n})v_{n}dx  \notag \\
&=&\int_{\mathbb{R}^{N}}\lambda V(x)v_{n}^{2}dx+\left( am\left( \left\Vert
u_{n}\right\Vert _{D^{1,2}}^{2}\right) +b\right) \left\Vert v_{n}\right\Vert
_{D^{1,2}}^{2}-\int_{\mathbb{R}^{N}}f(x,v_{n})v_{n}dx.  \label{20}
\end{eqnarray}%
In addition, it follows from (\ref{15}), conditions $(L1)$ and $(D3)$ that%
\begin{eqnarray*}
J_{\lambda ,a}\left( u_{0}\right) &=&\frac{a}{2}\widehat{m}\left( \left\Vert
u_{0}\right\Vert _{D^{1,2}}^{2}\right) +\frac{1}{2}\left( b\left\Vert
u_{0}\right\Vert _{D^{1,2}}^{2}+\int_{\mathbb{R}^{N}}\lambda
V(x)u_{0}^{2}dx\right) -\int_{\mathbb{R}^{N}}F(x,u_{0})dx \\
&&-\frac{1}{k}\left( \int_{\mathbb{R}^{N}}\lambda V(x)u_{0}^{2}dx+\left(
am\left( A\right) +b\right) \left\Vert u_{0}\right\Vert _{D^{1,2}}^{2}-\int_{%
\mathbb{R}^{N}}f(x,u_{0})u_{0}dx\right) \\
&=&\frac{a}{2}\left[ \widehat{m}\left( \left\Vert u_{0}\right\Vert
_{D^{1,2}}^{2}\right) -\frac{2}{k}m\left( A\right) \left\Vert
u_{0}\right\Vert _{D^{1,2}}^{2}\right] \\
&&+\frac{(k-2)\min \{b,1\}}{2k}\Vert u_{0}\Vert _{\lambda }^{2}+\int_{%
\mathbb{R}^{N}}\left[ \frac{1}{k}f(x,u_{0})u_{0}-F(x,u_{0})\right] dx \\
&\geq &\frac{a}{2}\left[ \widehat{m}\left( \left\Vert u_{0}\right\Vert
_{D^{1,2}}^{2}\right) -\frac{2}{k}m\left( A\right) \left\Vert
u_{0}\right\Vert _{D^{1,2}}^{2}\right] \\
&\geq &\frac{a}{2}\left[ \widehat{m}\left( \left\Vert u_{0}\right\Vert
_{D^{1,2}}^{2}\right) -\frac{2}{k}m\left( \left\Vert u_{0}\right\Vert
_{D^{1,2}}^{2}\right) \left\Vert u_{0}\right\Vert _{D^{1,2}}^{2}\right] \\
&&+\frac{a}{k}\left\Vert u_{0}\right\Vert _{D^{1,2}}^{2}\left[ m\left(
\left\Vert u_{0}\right\Vert _{D^{1,2}}^{2}\right) -m\left( A\right) \right]
\\
&\geq &\frac{a}{k}\left\Vert u_{0}\right\Vert _{D^{1,2}}^{2}\left[ m\left(
\left\Vert u_{0}\right\Vert _{D^{1,2}}^{2}\right) -m\left( A\right) \right] .
\end{eqnarray*}%
Then there exists a constant $K$ satisfying $K=0$ if $m\left( \left\Vert
u_{0}\right\Vert _{D^{1,2}}^{2}\right) \geq m\left( A\right) $ or $K<0$ if $%
m\left( \left\Vert u_{0}\right\Vert _{D^{1,2}}^{2}\right) <m\left( A\right) $
such that%
\begin{equation*}
J_{\lambda ,a}\left( u_{0}\right) \geq K.
\end{equation*}%
Using this, together with (\ref{16}), (\ref{20}), conditions $(L1)$ and $%
(D3) $ leads to%
\begin{eqnarray*}
D-K &\geq &\alpha -J_{a,\lambda }\left( u_{0}\right) =J_{a,\lambda }\left(
u_{n}\right) -J_{a,\lambda }\left( u_{0}\right) +o\left( 1\right) \\
&=&\frac{1}{2}\int_{\mathbb{R}^{N}}\lambda V(x)v_{n}^{2}dx+\frac{a}{2}\left[
\widehat{m}\left( \left\Vert u_{n}\right\Vert _{D^{1,2}}^{2}\right) -%
\widehat{m}\left( \left\Vert u_{0}\right\Vert _{D^{1,2}}^{2}\right) \right]
\\
&&+\frac{b}{2}\left\Vert v_{n}\right\Vert _{D^{1,2}}^{2}-\int_{\mathbb{R}%
^{N}}F(x,v_{n})dx+o\left( 1\right) \\
&=&\frac{k-2}{2k}\left( b\left\Vert v_{n}\right\Vert _{D^{1,2}}^{2}+\int_{%
\mathbb{R}^{N}}\lambda V(x)v_{n}^{2}dx\right) -\int_{\mathbb{R}^{N}}\left(
F(x,v_{n})-\frac{1}{k}f(x,v_{n})v_{n}\right) dx \\
&&+\frac{a}{2}\left[ \widehat{m}\left( \left\Vert u_{n}\right\Vert
_{D^{1,2}}^{2}\right) -\widehat{m}\left( \left\Vert u_{0}\right\Vert
_{D^{1,2}}^{2}\right) -\frac{2}{k}m\left( \left\Vert u_{n}\right\Vert
_{D^{1,2}}^{2}\right) \left\Vert v_{n}\right\Vert _{D^{1,2}}^{2}\right]
+o\left( 1\right) \\
&\geq &\frac{(k-2)\min \left\{ 1,b\right\} }{2k}\left\Vert v_{n}\right\Vert
_{\lambda }^{2}+o(1) \\
&=&\frac{\left( k-2\right) \min \left\{ 1,b\right\} }{2k}\left\Vert
v_{n}\right\Vert _{\lambda }^{2}+o(1),
\end{eqnarray*}%
which implies that there exists a constant $\widetilde{D}=\widetilde{D}(D)>0$
such that
\begin{equation}
\left\Vert v_{n}\right\Vert _{\lambda }^{2}\leq \frac{2k\widetilde{D}}{%
\left( k-2\right) \min \left\{ 1,b\right\} }+o(1).  \label{22}
\end{equation}%
It follows from the (\ref{3.1}), (\ref{17}), (\ref{20}) and (\ref{22}) that%
\begin{eqnarray*}
o\left( 1\right) &=&\int_{\mathbb{R}^{N}}\lambda V(x)v_{n}^{2}dx+am\left(
\left\Vert u_{n}\right\Vert _{D^{1,2}}^{2}\right) \left\Vert
v_{n}\right\Vert _{D^{1,2}}^{2}+b\left\Vert v_{n}\right\Vert
_{D^{1,2}}^{2}-\int_{\mathbb{R}^{N}}f(x,v_{n})v_{n}dx \\
&\geq &\min \{1,b\}\left\Vert v_{n}\right\Vert _{\lambda }^{2}-|q|_{\infty
}\int_{\mathbb{R}^{N}}|v_{n}|^{k}dx \\
&\geq &\min \{1,b\}\left\Vert v_{n}\right\Vert _{\lambda }^{2}-|q|_{\infty
}\Pi _{\lambda ,k}\left\Vert v_{n}\right\Vert _{\lambda }^{k} \\
&\geq &\min \{1,b\}\left\Vert v_{n}\right\Vert _{\lambda }^{2}-|q|_{\infty
}\Pi _{\lambda ,k}\left[ \frac{2k\widetilde{D}}{\left( k-2\right) \min
\left\{ 1,b\right\} }\right] ^{k/2}+o(1),
\end{eqnarray*}%
which implies that there exists $\widetilde{{\Lambda }}:=\widetilde{{\Lambda
}}(a,D)\geq \Lambda _{N}$ such that for $\lambda >\widetilde{{\Lambda }},$%
\begin{equation*}
v_{n}\rightarrow 0\text{ strongly in }X_{\lambda }.
\end{equation*}%
This completes the proof.
\end{proof}

\begin{proposition}
\label{m3} Suppose that $N\geq 3,$ conditions $(V1)-(V2),(L2),(L3)$ with $%
\delta \geq \frac{2}{N-2}$and $(D1)-(D2)$ hold. Then for each $D>0$ there
exists $\widetilde{{\Lambda }}_{1}=\widetilde{{\Lambda }}_{1}(a,D)\geq
\Lambda _{N}>0$ such that $J_{\lambda ,a}$ satisfies the $(C)_{\alpha }$%
--condition in $X_{\lambda }$ for all $\alpha <D$ and $\lambda >\widetilde{{%
\Lambda }}_{1}.$
\end{proposition}

\begin{proof}
The proof is similar to that of Proposition \ref{m2}, in which only some
places are adjusted. Now we briefly verify it. By Lemma \ref{h0}, there
exists a constant $K<0$ such that%
\begin{equation}
J_{\lambda ,a}\left( u_{0}\right) \geq K.  \label{21}
\end{equation}%
It follows from (\ref{16}), (\ref{20}), (\ref{21}), conditions $(L2)$ and $%
(D2)$ that%
\begin{eqnarray*}
D-K &\geq &\alpha -J_{a,\lambda }\left( u_{0}\right) =J_{a,\lambda }\left(
u_{n}\right) -J_{a,\lambda }\left( u_{0}\right) +o\left( 1\right)  \\
&=&\frac{1}{2}\int_{\mathbb{R}^{N}}\lambda V(x)v_{n}^{2}dx+\frac{a}{2}\left[
\widehat{m}\left( \left\Vert u_{n}\right\Vert _{D^{1,2}}^{2}\right) -%
\widehat{m}\left( \left\Vert u_{0}\right\Vert _{D^{1,2}}^{2}\right) \right]
\\
&&+\frac{b}{2}\left\Vert v_{n}\right\Vert _{D^{1,2}}^{2}-\int_{\mathbb{R}%
^{N}}F(x,v_{n})dx+o\left( 1\right)  \\
&=&\frac{k-2}{2k}\left( b\left\Vert v_{n}\right\Vert _{D^{1,2}}^{2}+\int_{%
\mathbb{R}^{N}}\lambda V(x)v_{n}^{2}dx\right) -\int_{\mathbb{R}^{N}}\left(
F(x,v_{n})-\frac{1}{k}f(x,v_{n})v_{n}\right) dx \\
&&+\frac{a}{2}\left[ \widehat{m}\left( \left\Vert u_{n}\right\Vert
_{D^{1,2}}^{2}\right) -\widehat{m}\left( \left\Vert u_{0}\right\Vert
_{D^{1,2}}^{2}\right) -\frac{2}{k}m\left( \left\Vert u_{n}\right\Vert
_{D^{1,2}}^{2}\right) \left\Vert v_{n}\right\Vert _{D^{1,2}}^{2}\right]
+o\left( 1\right)  \\
&\geq &\frac{\left( k-2\right) \min \left\{ 1,b\right\} }{2k}\left\Vert
v_{n}\right\Vert _{\lambda }^{2} \\
&&-\frac{a}{2}\widehat{m}\left( \left\Vert u_{0}\right\Vert
_{D^{1,2}}^{2}\right) -\frac{a}{k}m\left( \left\Vert u_{n}\right\Vert
_{D^{1,2}}^{2}\right) \left\Vert u_{n}\right\Vert _{D^{1,2}}^{2}+o(1) \\
&\geq &\frac{\left( k-2\right) \min \left\{ 1,b\right\} }{2k}\left\Vert
v_{n}\right\Vert _{\lambda }^{2} \\
&&-\frac{a}{2}m\left( \left\Vert u_{0}\right\Vert _{D^{1,2}}^{2}\right)
\left\Vert u_{0}\right\Vert _{D^{1,2}}^{2}-\frac{a}{k}m\left( A^{2}\right)
A^{2}+o(1),
\end{eqnarray*}%
which implies that for each $\lambda >\Lambda _{N}$ there exists a constant $%
\overline{D}=\overline{D}(D)>0$ such that
\begin{equation*}
\left\Vert v_{n}\right\Vert _{\lambda }^{2}\leq \frac{2k\overline{D}}{\left(
k-2\right) \min \left\{ 1,b\right\} }+o(1).
\end{equation*}%
Next, following the argument of Proposition \ref{m2}, we easily arrive at
the conclusion. This completes the proof.
\end{proof}

\begin{theorem}
\label{t9}Suppose that $N\geq 1,$ conditions $(V1)-(V2),(L1)$ and $(D1)-(D2)$
hold. Then for every $a>0$ and $\lambda >\widetilde{{\Lambda }},$ the energy
functional $J_{a,\lambda }$ has a nontrivial critical point $u_{\lambda }\in
X_{\lambda }$ such that $J_{a,\lambda }(u_{\lambda })>0.$
\end{theorem}

\begin{proof}
By Proposition \ref{m2} and $0<\eta \leq \alpha _{\lambda ,a}\leq \alpha
_{0,a}\left( \Omega \right) $ for all $\lambda \geq \Lambda _{N}$, $%
J_{\lambda ,a}$ satisfies the $(C)_{\alpha _{\lambda ,a}}$--condition in $%
X_{\lambda }$ for each $a>0$ and $\lambda >\widetilde{{\Lambda }}$. That is,
there exist a subsequence $\left\{ u_{n}\right\} $ and $u_{\lambda }\in
X_{\lambda }$ such that $u_{n}\rightarrow u_{\lambda }$ strongly in $%
X_{\lambda }.$ This implies that $u_{\lambda }$ is a nontrivial critical
point of $J_{\lambda ,a}$ satisfying $J_{\lambda ,a}\left( u_{\lambda
}\right) =\alpha _{\lambda ,a}>0.$
\end{proof}

\begin{theorem}
\label{t10}Suppose that $N\geq 3,$ conditions $(V1)-(V2),(L2),(L3)$ with $%
\delta \geq \frac{2}{N-2}$ and $(D1)-(D2)$ hold. Then there exists a
constant $\widetilde{{\Lambda }}_{2}\geq \max \left\{ \widetilde{{\Lambda }}%
_{0},\widetilde{{\Lambda }}_{1}\right\} $ such that for every $0<a<%
\widetilde{{a}}_{\ast }$ and $\lambda >\widetilde{{\Lambda }}_{2},$ the
energy functional $J_{a,\lambda }$ has a nontrivial critical point $%
u_{a,\lambda }^{+}\in X_{\lambda }$ satisfying $J_{a,\lambda }(u_{a,\lambda
}^{+})>0.$
\end{theorem}

\begin{proof}
Similar to the argument of Theorem \ref{t9}, by Proposition \ref{m3} we
easily arrive at the conclusion.
\end{proof}

\begin{lemma}
\label{h0}Suppose that $N\geq 3$, conditions $(V1)-{(V2)},(L3)$ with $\delta
\geq \frac{2}{N-2}$ and $(D1)-{(D2)}$ hold. Then the energy functional $%
J_{a,\lambda }$ is bounded below on $X_{\lambda }$ for all $a>0$ and%
\begin{equation*}
\lambda >\widetilde{{\Lambda }}_{3}:=\left\{
\begin{array}{ll}
\max \left\{ \Lambda _{N},\frac{2|q|_{\infty }}{c_{0}k}\left( \frac{2(\delta
+1)|q|_{\infty }}{m_{0}akS^{2^{\ast }}}\right) ^{\frac{k-2}{2^{\ast }-k}%
}\right\} & \text{ if }\delta =\frac{2}{N-2}, \\
\Lambda _{N} & \text{ if }\delta >\frac{2}{N-2}.%
\end{array}%
\right.
\end{equation*}%
Furthermore, if
\begin{equation*}
\lambda >\widetilde{{\Lambda }}_{4}:=\max \left\{ \widetilde{{\Lambda }}_{3},%
\frac{2|q|_{\infty }\left( 2^{\ast }-k\right) }{c_{0}k(2^{\ast }-2)}\right\}
,
\end{equation*}%
then there exists $\widetilde{{R}}_{a}>T_{0}^{1/2}$ such that%
\begin{equation*}
J_{a,\lambda }(u)\geq 0\text{ for all }u\in X_{\lambda }\text{ with }%
\left\Vert u\right\Vert _{\lambda }\geq \widetilde{{R}}_{a}.
\end{equation*}
\end{lemma}

\begin{proof}
If $\left\Vert u\right\Vert _{D^{1,2}}<T_{0}^{1/2},$ then by (\ref{10}), (%
\ref{3.1}) and the Young inequality one has%
\begin{eqnarray*}
J_{a,\lambda }\left( u\right) &\geq &\frac{\min \{b,1\}}{2}\left\Vert
u\right\Vert _{\lambda }^{2}-\frac{|q|_{\infty }}{k}\int_{\mathbb{R}%
^{N}}|u|^{k}dx \\
&\geq &\frac{\min \{b,1\}}{2}\left\Vert u\right\Vert _{\lambda }^{2}-\frac{%
|q|_{\infty }}{kS^{k}}\left\vert \left\{ V<c_{0}\right\} \right\vert
^{\left( 2^{\ast }-k\right) /2^{\ast }}\left\Vert u\right\Vert _{\lambda }^{%
\frac{2(2^{\ast }-k)}{2^{\ast }-2}}\left\Vert u\right\Vert _{D^{1,2}}^{\frac{%
2^{\ast }(k-2)}{2^{\ast }-2}} \\
&\geq &\frac{2^{\ast }(k-2)\min \{b,1\}}{2p(2^{\ast }-2)}\left\Vert
u\right\Vert _{\lambda }^{2} \\
&&-\frac{k-2}{k(2^{\ast }-2)\left( \min \{b,1\}\right) ^{\frac{2^{\ast }-k}{%
k-2}}}\left( |q|_{\infty }S^{-k}\left\vert \left\{ V<c_{0}\right\}
\right\vert ^{(2^{\ast }-k)/2^{\ast }}\right) \left\Vert u\right\Vert
_{D^{1,2}}^{2^{\ast }} \\
&\geq &\frac{2^{\ast }(k-2)\min \{b,1\}}{2p(2^{\ast }-2)}\left\Vert
u\right\Vert _{\lambda }^{2} \\
&&-\frac{k-2}{k(2^{\ast }-2)\left( \min \{b,1\}\right) ^{\frac{2^{\ast }-k}{%
k-2}}}\left( |q|_{\infty }S^{-k}\left\vert \left\{ V<c\right\} \right\vert
^{(2^{\ast }-k)/2^{\ast }}\right) ^{\frac{2^{\ast }-2}{k-2}}T_{0}^{2^{\ast
}/2},
\end{eqnarray*}%
which shows that $J_{a,\lambda }$ is bounded below on $X_{\lambda }$ for all
$a>0$ and $\lambda >\Lambda _{N}.$

If $\left\Vert u\right\Vert _{D^{1,2}}\geq T_{0}^{1/2},$ then we consider
two cases as follows:

$\left( i\right) $ $\delta =\frac{2}{N-2}$ and $\left\Vert u\right\Vert
_{D^{1,2}}\geq T_{0}^{1/2}:$ The argument is also divided into two seperate
cases:\newline
Case $A:\int_{\mathbb{R}^{N}}\lambda V(x)u^{2}dx\geq \lambda
c_{0}S^{-2^{\ast }}\left( \frac{2|q|_{\infty }}{k\lambda c_{0}}\right) ^{%
\frac{2^{\ast }-2}{k-2}}\left\Vert u\right\Vert _{D^{1,2}}^{2^{\ast }}.$ By
condition $(L3)$, (\ref{10}), (\ref{12}) and (\ref{3.2}), for each $\lambda
>0$ one has%
\begin{eqnarray*}
&&J_{a,\lambda }\left( u\right)  \\
&\geq &\frac{a}{2}\widehat{m}\left( \left\Vert u\right\Vert
_{D^{1,2}}^{2}\right) +\frac{1}{2}\left( b\left\Vert u\right\Vert
_{D^{1,2}}^{2}+\int_{\mathbb{R}^{N}}\lambda V(x)u^{2}dx\right) -\frac{%
|q|_{\infty }}{k}\int_{\mathbb{R}^{N}}|u|^{k}dx \\
&\geq &\frac{m_{0}a}{2(\delta +1)}\left\Vert u\right\Vert
_{D^{1,2}}^{2(\delta +1)}+\frac{1}{2}\left( b\left\Vert u\right\Vert
_{D^{1,2}}^{2}+\int_{\mathbb{R}^{N}}\lambda V(x)u^{2}dx\right)  \\
&&-\frac{|q|_{\infty }}{k}\left( \frac{1}{\lambda c_{0}}\int_{\mathbb{R}%
^{N}}\lambda V(x)u^{2}dx+S^{-2}\left\vert \left\{ V<c_{0}\right\}
\right\vert ^{2/N}\left\Vert u\right\Vert _{D^{1,2}}^{2}\right) ^{\frac{%
2^{\ast }-k}{2^{\ast }-2}}\left( S^{-1}\left\Vert u\right\Vert
_{D^{1,2}}\right) ^{\frac{2^{\ast }(k-2)}{2^{\ast }-2}} \\
&\geq &\frac{m_{0}a}{2^{\ast }}\left\Vert u\right\Vert _{D^{1,2}}^{2^{\ast
}}+\frac{b}{2}\left\Vert u\right\Vert _{D^{1,2}}^{2}-\frac{|q|_{\infty
}\left\vert \left\{ V<c_{0}\right\} \right\vert ^{1-k/2^{\ast }}}{kS^{k}}%
\left\Vert u\right\Vert _{D^{1,2}}^{k} \\
&&+\frac{1}{2}\left( \int_{\mathbb{R}^{N}}\lambda V(x)u^{2}dx\right) ^{\frac{%
2^{\ast }-k}{2^{\ast }-2}}\left[ \left( \int_{\mathbb{R}^{N}}\lambda
V(x)u^{2}dx\right) ^{\frac{k-2}{2^{\ast }-2}}-\frac{2|q|_{\infty }}{k\left(
\lambda c_{0}\right) ^{\frac{2^{\ast }-k}{2^{\ast }-2}}}\left(
S^{-1}\left\Vert u\right\Vert _{D^{1,2}}\right) ^{\frac{2^{\ast }(k-2)}{%
2^{\ast }-2}}\right]  \\
&\geq &\frac{m_{0}a}{2^{\ast }}\left\Vert u\right\Vert _{D^{1,2}}^{2^{\ast
}}+\frac{b}{2}\left\Vert u\right\Vert _{D^{1,2}}^{2}-\frac{|q|_{\infty
}\left\vert \left\{ V<c_{0}\right\} \right\vert ^{1-k/2^{\ast }}}{kS^{k}}%
\left\Vert u\right\Vert _{D^{1,2}}^{k},
\end{eqnarray*}%
where we have used the Young and Sobolev inequalities. This implies that $%
J_{a,\lambda }\left( u\right) $ is bounded below on $X$ for all $a>0$ and $%
\lambda >\Lambda _{N}.$

Case $B:\int_{\mathbb{R}^{N}}\lambda V(x)u^{2}dx<\lambda c_{0}S^{-2^{\ast
}}\left( \frac{2|q|_{\infty }}{k\lambda c_{0}}\right) ^{\frac{2^{\ast }-2}{%
k-2}}\left\Vert u\right\Vert _{D^{1,2}}^{2^{\ast }}.$ It follows from (\ref%
{10}) that%
\begin{eqnarray*}
&&\int_{\mathbb{R}^{N}}|u|^{k}dx \\
&\leq &\left( \frac{1}{\lambda c_{0}}\int_{\mathbb{R}^{N}}\lambda
V(x)u^{2}dx+S^{-2}\left\vert \left\{ V<c_{0}\right\} \right\vert
^{2/N}\left\Vert u\right\Vert _{D^{1,2}}^{2}\right) ^{\frac{2^{\ast }-k}{%
2^{\ast }-2}}\left( S^{-1}\left\Vert u\right\Vert _{D^{1,2}}\right) ^{\frac{%
2^{\ast }(k-2)}{2^{\ast }-2}} \\
&\leq &\left[ S^{-2^{\ast }}\left( \frac{2|q|_{\infty }}{k\lambda c_{0}}%
\right) ^{\frac{2^{\ast }-2}{k-2}}\left\Vert u\right\Vert
_{D^{1,2}}^{2^{\ast }}+S^{-2}\left\vert \left\{ V<c_{0}\right\} \right\vert
^{2/N}\left\Vert u\right\Vert _{D^{1,2}}^{2}\right] ^{\frac{2^{\ast }-k}{%
2^{\ast }-2}}\left( S^{-1}\left\Vert u\right\Vert _{D^{1,2}}\right) ^{\frac{%
2^{\ast }(k-2)}{2^{\ast }-2}} \\
&\leq &S^{-2^{\ast }}\left( \frac{2|q|_{\infty }}{k\lambda c_{0}}\right) ^{%
\frac{2^{\ast }-2}{k-2}}\left\Vert u\right\Vert _{D^{1,2}}^{2^{\ast }}+S^{-%
\frac{2^{\ast }(k-2)}{2^{\ast }-2}}\left( S^{-2}\left\vert \left\{
V<c_{0}\right\} \right\vert ^{2/N}\right) ^{\frac{2^{\ast }-k}{2^{\ast }-2}%
}\left\Vert u\right\Vert _{D^{1,2}}^{k} \\
&=&S^{-2^{\ast }}\left( \frac{2|q|_{\infty }}{k\lambda c_{0}}\right) ^{\frac{%
2^{\ast }-2}{k-2}}\left\Vert u\right\Vert _{D^{1,2}}^{2^{\ast
}}+S^{-k}\left\vert \left\{ V<c\right\} \right\vert ^{\frac{2^{\ast }-k}{%
2^{\ast }}}\left\Vert u\right\Vert _{D^{1,2}}^{k}.
\end{eqnarray*}%
Using this, together with condition $(L3)$ once again, yields%
\begin{eqnarray*}
J_{a,\lambda }\left( u\right)  &\geq &\frac{a}{2}\widehat{m}\left(
\left\Vert u\right\Vert _{D^{1,2}}^{2}\right) +\frac{1}{2}\left( b\left\Vert
u\right\Vert _{D^{1,2}}^{2}+\int_{\mathbb{R}^{N}}\lambda V(x)u^{2}dx\right) -%
\frac{|q|_{\infty }}{k}\int_{\mathbb{R}^{N}}|u|^{k}dx \\
&\geq &\frac{m_{0}a}{2(\delta +1)}\left\Vert u\right\Vert
_{D^{1,2}}^{2(\delta +1)}+\frac{1}{2}\left( b\left\Vert u\right\Vert
_{D^{1,2}}^{2}+\int_{\mathbb{R}^{N}}\lambda V(x)u^{2}dx\right)  \\
&&-\frac{|q|_{\infty }}{k}\left[ S^{-2^{\ast }}\left( \frac{2||_{\infty }}{%
k\lambda c_{0}}\right) ^{\frac{2^{\ast }-k}{k-2}}\left\Vert u\right\Vert
_{D^{1,2}}^{2^{\ast }}+S^{-k}\left\vert \left\{ V<c_{0}\right\} \right\vert
^{\frac{2^{\ast }-k}{2^{\ast }}}\left\Vert u\right\Vert _{D^{1,2}}^{k}\right]
\\
&\geq &\left[ \frac{m_{0}a}{2(\delta +1)}-\frac{|q|_{\infty }}{kS^{2^{\ast }}%
}\left( \frac{2|q|_{\infty }}{p\lambda c_{0}}\right) ^{\frac{2^{\ast }-k}{k-2%
}}\right] \left\Vert u\right\Vert _{D^{1,2}}^{2^{\ast }} \\
&&-\frac{|q|_{\infty }}{kS^{2^{\ast }}}\left\vert \left\{ V<c_{0}\right\}
\right\vert ^{\frac{2^{\ast }-k}{2^{\ast }}}\left\Vert u\right\Vert
_{D^{1,2}}^{k}.
\end{eqnarray*}%
This shows that if%
\begin{equation*}
\lambda >\max \left\{ \Lambda _{N},\frac{2|q|_{\infty }}{kc_{0}}\left( \frac{%
2(\delta +1)|q|_{\infty }}{m_{0}akS^{2^{\ast }}}\right) ^{\frac{k-2}{2^{\ast
}-k}}\right\} ,
\end{equation*}%
then $J_{a,\lambda }$ is bounded below on $X_{\lambda }$ for all $a>0$ and
there exists $R_{a}>0$ such that
\begin{equation*}
J_{a,\lambda }(u)\geq 0\text{ for all }u\in X_{\lambda }\text{ with }%
\left\Vert u\right\Vert _{D^{1,2}}\geq R_{a}.
\end{equation*}%
$\left( ii\right) $ $\delta >\frac{2}{N-2}$ and $\left\Vert u\right\Vert
_{D^{1,2}}\geq T_{0}^{1/2}:$ It follows from condition $(L3)$, (\ref{10}), (%
\ref{3.1}) and the Young inequality that%
\begin{eqnarray*}
J_{a,\lambda }\left( u\right)  &=&\frac{a}{2}\widehat{m}\left( \left\Vert
u\right\Vert _{D^{1,2}}^{2}\right) +\frac{1}{2}\left( b\left\Vert
u\right\Vert _{D^{1,2}}^{2}+\int_{\mathbb{R}^{N}}\lambda V(x)u^{2}dx\right)
-\int_{\mathbb{R}^{N}}F(x,u)dx \\
&\geq &\frac{m_{0}a}{2(\delta +1)}\left\Vert u\right\Vert
_{D^{1,2}}^{2(\delta +1)}+\frac{\min \{b,1\}}{2}\left\Vert u\right\Vert
_{\lambda }^{2}-\frac{|q|_{\infty }}{k}\int_{\mathbb{R}^{N}}|u|^{k}dx \\
&\geq &\frac{m_{0}a}{2(\delta +1)}\left\Vert u\right\Vert
_{D^{1,2}}^{2(\delta +1)}+\frac{\min \{b,1\}}{2}\left\Vert u\right\Vert
_{\lambda }^{2} \\
&&-\frac{|q|_{\infty }}{k}S^{-k}\left\vert \left\{ V<c_{0}\right\}
\right\vert ^{\left( 2^{\ast }-k\right) /2^{\ast }}\left\Vert u\right\Vert
_{\lambda }^{\frac{2(2^{\ast }-k)}{2^{\ast }-2}}\left\Vert u\right\Vert
_{D^{1,2}}^{\frac{2^{\ast }(k-2)}{2^{\ast }-2}} \\
&\geq &\frac{m_{0}a}{2(\delta +1)}\left\Vert u\right\Vert
_{D^{1,2}}^{2(\delta +1)}+\min \{b,1\}\frac{2^{\ast }(k-2)}{2p(2^{\ast }-2)}%
\left\Vert u\right\Vert _{\lambda }^{2} \\
&&-\frac{k-2}{k(2^{\ast }-2)\left( \min \{b,1\}\right) ^{\frac{2^{\ast }-k}{%
k-2}}}\left( |q|_{\infty }S^{-k}\left\vert \left\{ V<c_{0}\right\}
\right\vert ^{(2^{\ast }-k)/2^{\ast }}\right) \left\Vert u\right\Vert
_{D^{1,2}}^{2^{\ast }} \\
&\geq &\frac{m_{0}a}{2(\delta +1)}\left\Vert u\right\Vert
_{D^{1,2}}^{2(\delta +1)} \\
&&-\frac{k-2}{k(2^{\ast }-2)\left( \min \{b,1\}\right) ^{\frac{2^{\ast }-k}{%
k-2}}}\left( |q|_{\infty }S^{-k}\left\vert \left\{ V<c\right\} \right\vert
^{(2^{\ast }-k)/2^{\ast }}\right) ^{\frac{2^{\ast }-2}{k-2}}\left\Vert
u\right\Vert _{D^{1,2}}^{2^{\ast }},
\end{eqnarray*}%
which implies that $J_{a,\lambda }\left( u\right) $ is bounded below on $X$
for all $a>0$ and $\lambda >\Lambda _{N},$ since $\delta >\frac{2}{N-2}.$
Moreover, for every $a>0,$ there exists%
\begin{equation*}
R_{a}>t_{\overline{B}}:=\left[ \frac{2(\delta +1)(k-2)\left( |q|_{\infty
}S^{-k}\left\vert \left\{ V<c_{0}\right\} \right\vert ^{(2^{\ast
}-k)/2^{\ast }}\right) ^{\frac{2^{\ast }-2}{k-2}}}{km_{0}a(2^{\ast
}-2)\left( \min \{b,1\}\right) ^{\frac{2^{\ast }-k}{k-2}}}\right] ^{\frac{1}{%
2\delta +2-2^{\ast }}}
\end{equation*}%
such that%
\begin{equation*}
J_{a,\lambda }(u)\geq 0\text{ for all }u\in X_{\lambda }\text{ with }%
\left\Vert u\right\Vert _{D^{1,2}}\geq \overline{R}_{a}=\max \left\{
T_{0}^{1/2},R_{a}\right\} .
\end{equation*}

Next, we show that there exists a constant $\widehat{R}_{a}>\overline{R}_{a}$
such that%
\begin{equation*}
J_{a,\lambda }(u)\geq 0\text{ for all }u\in X_{\lambda }\text{ with }%
\left\Vert u\right\Vert _{\lambda }\geq \widehat{R}_{a}.
\end{equation*}%
Let
\begin{equation}
\widetilde{{R}}_{a}=\left[ \overline{R}_{a}^{2}+2A_{0}\left( \overline{R}%
_{a}\right) \left( 1-\frac{1}{\lambda c_{0}}\left( \frac{2\left( 2^{\ast
}-k\right) |q|_{\infty }}{k(2^{\ast }-2)}\right) \right) ^{-1}\right] ^{1/2},
\label{14-10}
\end{equation}%
where
\begin{equation*}
A_{0}\left( \overline{R}_{a}\right) =\frac{\left( 2^{\ast }-k\right)
|q|_{\infty }}{k(2^{\ast }-2)}\left\vert \left\{ V<c_{0}\right\} \right\vert
^{2/N}S^{-2}\overline{R}_{a}^{2}+\frac{(k-2)|q|_{\infty }}{k(2^{\ast }-2)}%
S^{-2^{\ast }}\overline{R}_{a}^{2^{\ast }}.
\end{equation*}%
For $u\in X_{\lambda }$ with $\left\Vert u\right\Vert _{\lambda }\geq
\widetilde{{R}}_{a}.$ If $\left\Vert u\right\Vert _{D^{1,2}}\geq \overline{R}%
_{a},$ then the result holds clearly. If $T_{0}^{1/2}\leq \left\Vert
u\right\Vert _{D^{1,2}}<\overline{R}_{a},$ then it is enough to indicate
that $J_{a,\lambda }(u)\geq 0$ when
\begin{equation*}
\int_{\mathbb{R}^{N}}\lambda V(x)u^{2}dx\geq 2A_{0}\left( \overline{R}%
_{a}\right) \left( 1-\frac{2\left( 2^{\ast }-k\right) |q|_{\infty }}{\lambda
c_{0}k(2^{\ast }-2)}\right) ^{-1}.
\end{equation*}%
Indeed, by (\ref{14-7}) we deduce that%
\begin{eqnarray*}
J_{a,\lambda }(u) &\geq &\frac{a}{2}\widehat{m}\left( \left\Vert
u\right\Vert _{D^{1,2}}^{2}\right) +\frac{1}{2}\left( b\left\Vert
u\right\Vert _{D^{1,2}}^{2}+\int_{\mathbb{R}^{N}}\lambda V(x)u^{2}dx\right) -%
\frac{|q|_{\infty }}{k}\int_{\mathbb{R}^{N}}|u|^{k}dx \\
&\geq &\frac{1}{2}\int_{\mathbb{R}^{N}}\lambda V(x)u^{2}dx-\frac{%
(k-2)|q|_{\infty }}{(2^{\ast }-2)k}S^{-2^{\ast }}\overline{R}_{a}^{2^{\ast }}
\\
&&-\frac{\left( 2^{\ast }-p\right) |q|_{\infty }}{k(2^{\ast }-2)}\left(
\frac{1}{\lambda c_{0}}\int_{\mathbb{R}^{N}}\lambda V(x)u^{2}dx+\left\vert
\left\{ V<c_{0}\right\} \right\vert ^{2/N}S^{-2}\overline{R}_{a}^{2}\right)
\\
&\geq &\frac{1}{2}\left[ 1-\frac{2\left( 2^{\ast }-k\right) |q|_{\infty }}{%
\lambda c_{0}k(2^{\ast }-2)}\right] \int_{\mathbb{R}^{N}}\lambda
V(x)u^{2}dx-A_{0}\left( \overline{R}_{a}\right) \\
&\geq &0.
\end{eqnarray*}%
Hence, we obtain that there exists a constant $\widetilde{{R}}_{a}>0$
defined as (\ref{14-10}) such that%
\begin{equation*}
J_{a,\lambda }(u)>0\text{ for all }u\in X_{\lambda }\text{ with }\left\Vert
u\right\Vert _{\lambda }\geq \widetilde{{R}}_{a}.
\end{equation*}%
This completes the proof.
\end{proof}

\begin{lemma}
\label{lem8}Suppose that $N\geq 3,$ conditions $(V1)-(V2),(L3)$ with $\delta
\geq \frac{2}{N-2}$ and $(D1)-(D2)$ hold. Then for every $a>0$ and $\lambda >%
\widetilde{{\Lambda }}_{4}$ one has
\begin{equation*}
\widetilde{{\theta }}_{a}=:\inf \left\{ J_{a,\lambda }(u):u\in X_{\lambda }%
\text{ with }\left\Vert u\right\Vert _{\lambda }<\widetilde{{R}}_{a}\right\}
<0.  \label{3.8}
\end{equation*}
\end{lemma}

\begin{proof}
The proof directly follows from Lemmas \ref{lem10} and \ref{h0}.
\end{proof}

\begin{theorem}
\label{t6}Suppose that $N\geq 3,$ conditions $(V1)-(V2),(L2),(L3)$ with $%
\delta \geq \frac{2}{N-2}$ and $(D1)-(D2)$ hold. Then there exists a
constant $\widetilde{{\Lambda }}_{5}\geq \max \left\{ \widetilde{{\Lambda }}%
_{1},\widetilde{{\Lambda }}_{4}\right\} $ such that for every $a>0$ and $%
\lambda \geq \widetilde{{\Lambda }}_{5},$ $J_{a,\lambda }$ has a nontrivial
critical point $u_{a,\lambda }^{-}\in X_{\lambda }$ such that $J_{a,\lambda
}(u_{a,\lambda }^{-})=\widetilde{{\theta }}_{a}<0.$
\end{theorem}

\begin{proof}
By Lemma \ref{lem8} and the Ekeland variational principle, there exists a
bounded minimizing sequence $\{u_{n}\}\subset X_{\lambda}$ with $\left\Vert
u_{n}\right\Vert _{\lambda }<\widetilde{{R}}_{a}$ such that%
\begin{equation*}
J_{a,\lambda }(u_{n})\rightarrow \widetilde{{\theta }}_{a}\quad \text{and}%
\quad (1+\Vert u_{n}\Vert _{X_{\lambda }})\Vert J_{a,\lambda }^{\prime
}(u_{n})\Vert _{X_{\lambda }^{-1}}\rightarrow 0\quad \text{as}\ n\rightarrow
\infty .
\end{equation*}%
According to Proposition \ref{m3}, there exist a subsequence $\left\{
u_{n}\right\} $ and $u_{a,\lambda }^{-}\in X_{\lambda }$ such that $%
u_{n}\rightarrow u_{a,\lambda }^{-}$ strongly in $X_{\lambda }.$ This
indicates that $u_{\lambda }$ is a nontrivial critical point of $J_{\lambda
,a}$ satisfying $J_{a,\lambda }(u_{a,\lambda }^{-})=\widetilde{{\theta }}%
_{a}<0.$
\end{proof}

\textbf{We are now ready to prove Theorems \ref{t1} and \ref{t2}: }Theorems %
\ref{t1} directly follows from Theorem \ref{t9}. By using Theorems \ref{t10}
and \ref{t6}, there exists a positive constant $\widetilde{{\Lambda }}_{\ast
}\geq \max \left\{ \widetilde{{\Lambda }}_{2},\widetilde{{\Lambda }}%
_{5}\right\} $ such that for every $0<a<\widetilde{{a}}_{\ast }$ and $%
\lambda \geq \widetilde{{\Lambda }}_{\ast },$ Eq. $(K_{a,\lambda })$ admits
two positive solutions $u_{a,\lambda }^{-}$ and $u_{a,\lambda }^{+}$
satisfying $J_{a,\lambda }(u_{a,\lambda }^{-})<0<J_{a,\lambda }(u_{a,\lambda
}^{+}).$ In particular, $u_{a,\lambda }^{-}$ is a ground state solution of
Eq. $(K_{a,\lambda }).$ Hence, we arrive at Theorem \ref{t2}.

\section{Proofs of Theorems \protect\ref{t3} and \protect\ref{t4}}

Following the argument at the begining of Section 4, by virtue of Lemmas \ref%
{lem2}, \ref{lem9} (or Lemma \ref{lem11}) and the mountain pass theorem \cite%
{E2}, we obtain that for each $\lambda \geq \Lambda _{N}$ and $0<a<\frac{1}{%
m_{\infty }\overline{\mu }_{1}^{(k)}}$ (or $0<a<\overline{a}_{\ast }$),
there exists a sequence $\left\{ u_{n}\right\} \subset X_{\lambda }$ such
that
\begin{equation}
J_{\lambda ,a}(u_{n})\rightarrow \alpha _{\lambda ,a}>0\quad \text{and}\quad
(1+\Vert u_{n}\Vert _{\lambda })\Vert J_{\lambda ,a}^{\prime }(u_{n})\Vert
_{X_{\lambda }^{-1}}\rightarrow 0,\quad \text{as}\ n\rightarrow \infty ,
\label{3.9}
\end{equation}%
where $0<\eta \leq \alpha _{\lambda ,a}\leq \alpha _{0,a}\left( \Omega
\right) <D_{a}.$ Furthermore, we have the following result.

\begin{lemma}
\label{lem6}Suppose that $N\geq 3,$ conditions $(V1)-(V2),(L3)$ with $\delta
>\frac{k-2}{2},(D1)^{\prime }$ and $(D2)$ hold. Then for each $0<a<\frac{1}{%
m_{\infty }\overline{\mu }_{1}^{(k)}},$ the sequence $\{u_{n}\}$ defined in (%
\ref{3.9}) is bounded in $X_{\lambda }$ for all $\lambda \geq \Lambda _{N}.$
\end{lemma}

\begin{proof}
Suppose on the contrary. Then $\Vert u_{n}\Vert _{\lambda }\rightarrow
+\infty $ as $n\rightarrow \infty $. We consider the proof in two separate
cases:\newline
Case $A:\Vert u_{n}\Vert _{D^{1,2}}\rightarrow \infty .$ It follows from (%
\ref{13}), (\ref{3.6}), condition $(L3)$ and the Caffarelli-Kohn-Nirenberg
inequality that%
\begin{eqnarray*}
o(1) &=&\frac{am\left( \left\Vert u_{n}\right\Vert _{D^{1,2}}^{2}\right)
\left\Vert u_{n}\right\Vert _{D^{1,2}}^{2}}{\Vert u_{n}\Vert _{D^{1,2}}^{k}}+%
\frac{b\left\Vert u_{n}\right\Vert _{D^{1,2}}^{2}+\int_{\mathbb{R}%
^{N}}\lambda V(x)u_{n}^{2}dx}{\Vert u_{n}\Vert _{D^{1,2}}^{k}}-\frac{\int_{%
\mathbb{R}^{N}}f(x,u_{n})u_{n}dx}{\Vert u_{n}\Vert _{D^{1,2}}^{k}} \\
&\geq &am_{0}\Vert u_{n}\Vert _{D^{1,2}}^{2(\delta +1)-k}-\frac{c_{\ast
}\left\Vert u_{n}\right\Vert _{D^{1,2}}^{k}}{\overline{\nu }_{1}^{(k)}\Vert
u_{n}\Vert _{D^{1,2}}^{k}} \\
&=&am_{0}\Vert u_{n}\Vert _{D^{1,2}}^{2(\delta +1)-k}-\frac{c_{\ast }}{%
\overline{\nu }_{1}^{(k)}} \\
&\rightarrow &\infty \text{ as }n\rightarrow \infty ,
\end{eqnarray*}%
since $2(\delta +1)>k.$ This is a contradiction.\newline
Case $B:\int_{\mathbb{R}^{N}}\lambda V(x)u_{n}^{2}dx\rightarrow \infty $ and
$\left\Vert u_{n}\right\Vert _{D^{1,2}}\leq C_{\ast }$ for some $C_{\ast }>0$
and for all $n.$ It follows from (\ref{13}), (\ref{3.6}) and condition $(L3)$
that%
\begin{eqnarray*}
o(1) &=&\frac{am\left( \left\Vert u_{n}\right\Vert _{D^{1,2}}^{2}\right)
\left\Vert u_{n}\right\Vert _{D^{1,2}}^{2}}{\int_{\mathbb{R}^{N}}\lambda
V(x)u_{n}^{2}dx}+\frac{b\left\Vert u_{n}\right\Vert _{D^{1,2}}^{2}+\int_{%
\mathbb{R}^{N}}\lambda V(x)u_{n}^{2}dx}{\int_{\mathbb{R}^{N}}\lambda
V(x)u_{n}^{2}dx}-\frac{\int_{\mathbb{R}^{N}}f(x,u_{n})u_{n}dx}{\int_{\mathbb{%
R}^{N}}\lambda V(x)u_{n}^{2}dx} \\
&\geq &1-\frac{c_{\ast }\left\Vert u_{n}\right\Vert _{D^{1,2}}^{k}}{%
\overline{\nu }_{1}^{(k)}\int_{\mathbb{R}^{N}}\lambda V(x)u_{n}^{2}dx}\geq 1-%
\frac{c_{\ast }C_{\ast }^{k}}{\overline{\nu }_{1}^{(k)}\int_{\mathbb{R}%
^{N}}\lambda V(x)u_{n}^{2}dx}=1+o\left( 1\right) .
\end{eqnarray*}%
This is a contradiction. In conclusion, the sequence $\{u_{n}\}$ is bounded
in $X_{\lambda }$ for all $0<a<\frac{1}{m_{\infty }\overline{\mu }_{1}^{(k)}}
$ and $\lambda \geq \Lambda _{N}.$ This completes the proof.
\end{proof}

Similar to Propositions \ref{m2} and \ref{m3}, we likewise obtain the
following two compactness lemmas for the functional $J_{\lambda ,a}$ under
conditions $(D1)^{\prime }$ and $(D2)$.

\begin{proposition}
\label{m1} Suppose that $N\geq 3,$ conditions $(V1)-(V3),(L1),(D1)^{\prime }$
and $(D2)$ hold. Then for each $D>0$ there exists $\overline{\Lambda }=%
\overline{\Lambda }(a,D)\geq \Lambda _{N}>0$ such that $J_{\lambda ,a}$
satisfies the $(C)_{\alpha }$--condition in $X_{\lambda }$ for all $\alpha
<D $ and $\lambda >\overline{\Lambda }.$
\end{proposition}

\begin{proposition}
\label{m4} Suppose that $N\geq 3,$ conditions $(V1)-(V3),(L2),(L3)$ with $%
\delta >\frac{k-2}{2},(D1)^{\prime }$ and $(D2)$ hold. Then for each $D>0$
there exists $\overline{\Lambda }_{\ast }=\overline{\Lambda }_{\ast
}(a,D)\geq \Lambda _{N}$ such that $J_{\lambda ,a}$ satisfies the $%
(C)_{\alpha }$--condition in $X_{\lambda }$ for all $\alpha <D$ and $\lambda
>\overline{\Lambda }_{\ast }.$
\end{proposition}

By Proposition \ref{m1} and \ref{m4}, we now give the following two
existence results.

\begin{theorem}
\label{t11}Suppose that $N\geq 3,$ conditions $(V1)-(V3),(L1),(D1)^{\prime }$
and $(D2)$ hold. Then for each $0<a<\frac{1}{m_{\infty }\overline{\mu }%
_{1}^{(k)}}$ the energy functional $J_{a,\lambda }$ admits a nontrivial
critical point $u_{\lambda }\in X_{\lambda }$ such that $J_{a,\lambda
}(u_{\lambda })>0$ for all $\lambda >\overline{\Lambda }.$
\end{theorem}

\begin{proof}
Similar to the proof of Theorem \ref{t9}, it is easily proved by using (\ref%
{3.9}), Lemma \ref{lem5} and Proposition \ref{m1}.
\end{proof}

\begin{theorem}
\label{t12}Suppose that $N\geq 3,$ conditions $(V1)-(V3),(L2),(L3)$ with $%
\delta >\frac{k-2}{2},(D1)^{\prime }$ and $(D2)$ hold. Then for every $0<a<%
\overline{a}_{\ast }$ and $\lambda >\overline{\Lambda }_{\ast },$ the energy
functional $J_{a,\lambda }$ has a nontrivial critical point $u_{a,\lambda
}^{+}\in X_{\lambda }$ satisfying $J_{a,\lambda }(u_{a,\lambda }^{+})>0.$
\end{theorem}

\begin{proof}
Similar to the argument of Theorem \ref{t10}, we easily arrive at the
conclusion by Lemma \ref{lem6} and Proposition \ref{m4}.
\end{proof}

\begin{lemma}
\label{lem4}Suppose that $N\geq 3$, conditions $(V1)-{(V3)},(L3)$ with $%
\delta >\frac{k-2}{2}$, ${(D1)}^{\prime }$ and ${(D2)}$ hold. Then the
energy functional $J_{a,\lambda }$ is bounded below on $X_{\lambda }$ for
all $a>0$ and $\lambda >0$. Furthermore, there exists $\overline{R}_{a}>0$
such that%
\begin{equation*}
J_{a,\lambda }(u)\geq 0\text{ for all }u\in X_{\lambda }\text{ with }%
\left\Vert u\right\Vert _{\lambda }\geq \overline{R}_{a}.
\end{equation*}
\end{lemma}

\begin{proof}
If $\left\Vert u\right\Vert _{D^{1,2}}<T_{0}^{1/2},$ then by (\ref{12}) and (%
\ref{3.4}) one has%
\begin{eqnarray*}
J_{a,\lambda }\left( u\right)  &\geq &\frac{1}{2}\left( b\left\Vert
u\right\Vert _{D^{1,2}}^{2}+\int_{\mathbb{R}^{N}}\lambda V(x)u^{2}dx\right) -%
\frac{c_{\ast }}{k\overline{\nu }_{1}^{(k)}}\left\Vert u\right\Vert
_{D^{1,2}}^{k} \\
&\geq &\frac{b}{2}\left\Vert u\right\Vert _{D^{1,2}}^{2}-\frac{c_{\ast }}{k%
\overline{\nu }_{1}^{(k)}}\left\Vert u\right\Vert _{D^{1,2}}^{k},
\end{eqnarray*}%
which implies that $J_{a,\lambda }$ is bounded below on $X_{\lambda }$ for
all $a>0$ and $\lambda >0.$

If $\left\Vert u\right\Vert _{D^{1,2}}\geq T_{0}^{1/2},$ then it follows
from condition $(L3)$ with $\delta >\frac{k-2}{2},$ (\ref{12}) and (\ref{3.4}%
) that%
\begin{eqnarray*}
J_{a,\lambda }\left( u\right)  &\geq &\frac{a}{2}\widehat{m}\left(
\left\Vert u\right\Vert _{D^{1,2}}^{2}\right) +\frac{1}{2}\left( b\left\Vert
u\right\Vert _{D^{1,2}}^{2}+\int_{\mathbb{R}^{N}}\lambda V(x)u^{2}dx\right) -%
\frac{c_{\ast }}{k\overline{\nu }_{1}^{(k)}}\left\Vert u\right\Vert
_{D^{1,2}}^{k} \\
&\geq &\frac{m_{0}a}{2(\delta +1)}\left\Vert u\right\Vert
_{D^{1,2}}^{2(\delta +1)}-\frac{c_{\ast }}{k\overline{\nu }_{1}^{(k)}}%
\left\Vert u\right\Vert _{D^{1,2}}^{k},
\end{eqnarray*}%
which implies that $J_{a,\lambda }\left( u\right) $ is bounded below on $X$
for all $a>0$ and $\lambda >0,$ since $\delta >\frac{k}{2}-1.$ Moreover, for
every $a>0,$ there exists $R_{a}>t_{\overline{B}}:=\left( \frac{2(\delta
+1)c_{\ast }}{k\overline{\nu }_{1}^{(k)}m_{0}a}\right) ^{1/(2\delta +2-k)}$
such that%
\begin{equation*}
J_{a,\lambda }(u)\geq 0\text{ for all }u\in X_{\lambda }\text{ with }%
\left\Vert u\right\Vert _{D^{1,2}}\geq R_{a}.
\end{equation*}

Next, we show that there exists a constant $\overline{R}_{a}>0$ such that $%
J_{a,\lambda }(u)\geq 0$ for all $u\in X_{\lambda }$ with $\left\Vert
u\right\Vert _{\lambda }\geq \overline{R}_{a}.$ Let
\begin{equation}
\overline{R}_{a}=\left( R_{a}^{2}+\frac{2c_{\ast }}{k\overline{\nu }%
_{1}^{(k)}}R_{a}^{k}\right) ^{1/2}>0.  \label{14-8}
\end{equation}%
For $u\in X_{\lambda }$ with $\left\Vert u\right\Vert _{\lambda }\geq
\overline{R}_{a}.$ If $\left\Vert u\right\Vert _{D^{1,2}}\geq R_{a},$ then
the result holds clearly. If $\left\Vert u\right\Vert _{D^{1,2}}<R_{a},$
then it is enough to indicate that $J_{a,\lambda }(u)\geq 0$ when $\int_{%
\mathbb{R}^{N}}\lambda V(x)u^{2}dx\geq \frac{2c_{\ast }}{k\overline{\nu }%
_{1}^{(k)}}R_{a}^{k}.$ Indeed, we have%
\begin{eqnarray*}
J_{a,\lambda }(u) &\geq &\frac{a}{2}\widehat{m}\left( \left\Vert
u\right\Vert _{D^{1,2}}^{2}\right) +\frac{1}{2}\left( b\left\Vert
u\right\Vert _{D^{1,2}}^{2}+\int_{\mathbb{R}^{N}}\lambda V(x)u^{2}dx\right) -%
\frac{c_{\ast }}{k\overline{\nu }_{1}^{(k)}}\left\Vert u\right\Vert
_{D^{1,2}}^{k} \\
&\geq &\frac{1}{2}\int_{\mathbb{R}^{N}}\lambda V(x)u^{2}dx-\frac{c_{\ast }}{k%
\overline{\nu }_{1}^{(k)}}R_{a}^{k}\geq 0.
\end{eqnarray*}%
Hence, we obtain that there exists a constant $\overline{R}_{a}>0$ defined
as (\ref{14-8}) such that%
\begin{equation*}
J_{a,\lambda }(u)>0\text{ for all }u\in X_{\lambda }\text{ with }\left\Vert
u\right\Vert _{\lambda }\geq \overline{R}_{a}.
\end{equation*}%
This completes the proof.
\end{proof}

\begin{lemma}
\label{lem12}Suppose that $N\geq 3,$ conditions $(V1)-(V3),(L3)$ with $%
\delta >\frac{k-2}{2},(D1)^{\prime }$ and $(D2)$ hold. Then for every $a>0$
and $\lambda >0$ one has
\begin{equation}
\overline{\theta }_{a}=:\inf \left\{ J_{a,\lambda }(u):u\in X_{\lambda }%
\text{ with }\left\Vert u\right\Vert _{\lambda }<\overline{R}_{a}\right\} <0.
\label{3.10}
\end{equation}
\end{lemma}

\begin{proof}
The proof directly follows from Lemmas \ref{lem11} and \ref{lem4}.
\end{proof}

\begin{theorem}
\label{t5}Suppose that $N\geq 3$, conditions $(V1)-(V3),(L2),(L3)$ with $%
\delta >\frac{k-2}{2},(D1)^{\prime }$ and $(D2)$ hold.. Then for every $a>0$
and $\lambda >\overline{\Lambda }_{\ast },$ $J_{a,\lambda }$ has a nonzero
critical point $u_{a,\lambda }^{-}\in X_{\lambda }$ such that
\begin{equation*}
J_{a,\lambda }(u_{a,\lambda }^{-})=\overline{\theta }_{a}<0,
\end{equation*}%
where $\overline{\theta }_{a}$ is as in (\ref{3.10}).
\end{theorem}

\begin{proof}
By Lemma \ref{lem12} and the Ekeland variational principle, we obtain that
there exists a minimizing bounded sequence $\{u_{n}\}\subset X_{\lambda }$
with $\Vert u_{n}\Vert _{\lambda }<\overline{R}_{a}$ such that%
\begin{equation*}
J_{a,\lambda }(u_{n})\rightarrow \overline{\theta }_{a}\text{ and }%
J_{a,\lambda }^{\prime }(u_{n})\rightarrow 0\text{ as }n\rightarrow \infty .
\end{equation*}%
Then from Proposition \ref{m4} it follows that there exist a subsequence $%
\{u_{n}\}$ and $u_{a,\lambda }^{-}\in X_{\lambda }$ with $\Vert u_{a,\lambda
}^{-}\Vert _{\lambda }<\overline{R}_{a}$ such that $u_{n}\rightarrow
u_{a,\lambda }^{-}$ strongly in $X_{\lambda }.$ This indicates that $%
J_{a,\lambda }^{\prime }(u_{a,\lambda }^{-})=0$ and $J_{a,\lambda
}(u_{a,\lambda }^{-})=\overline{\theta }_{a}<0.$ The proof is complete.
\end{proof}

\textbf{We are now ready to prove Theorems \ref{t3} and \ref{t4}: }Theorems %
\ref{t3} directly follows from Theorem \ref{t11}. By virtue of Theorems \ref%
{t12} and \ref{t5}, for every $0<a<\overline{a}_{\ast }$ and $\lambda >%
\overline{\Lambda }_{\ast },$ Eq. $(K_{a,\lambda })$ admits two positive
solutions $u_{a,\lambda }^{-}$ and $u_{a,\lambda }^{+}$ satisfying $%
J_{a,\lambda }(u_{a,\lambda }^{-})<0<J_{a,\lambda }(u_{a,\lambda }^{+}).$ In
particular, $u_{a,\lambda }^{-}$ is a ground state solution of Eq. $%
(K_{a,\lambda }).$ Hence, Theorem \ref{t4} is proved.

\end{document}